\documentclass[11pt]{amsart}

\usepackage[usenames,dvipsnames]{color}
\usepackage{latexsym,xspace,enumerate,amsfonts,amsmath,amssymb,amscd}
\usepackage{xypic,hyperref,syntonly,slashed,paralist}
\reversemarginpar
\usepackage[arrow, matrix, curve]{xy} 

\numberwithin{equation}{section}

\newcommand{\R}{\mathbb{R}}
\newcommand{\C}{\mathbb{C}}

\newcommand{\SL}{\mr{SL}}

\newcommand{\SU}{\mr{SU}}
\newcommand{\Sp}{\mr{Sp}}

\newcommand{\su}{\mathfrak{su}}

\newcommand{\mf}{\mathfrak}
\newcommand{\mr}{\mathrm}

\newcommand{\mc}{\mathcal}

\newcommand{\End}{\mathop{\rm End}\nolimits}

\renewcommand{\ker}{\mathop{\rm ker}\nolimits}
\newcommand{\im}{\mathop{\rm im}\nolimits}

\renewcommand{\Re}{\mathop{\rm Re}\nolimits}
\renewcommand{\Im}{\mathop{\rm Im}\nolimits}

\newcommand{\Tr}{\mathop{\rm Tr}\nolimits}

\newcommand{\delb}{\bar{\partial}}

\newcommand{\note}[1]{\marginpar{\raggedright\if@twoside\ifodd\c@page\raggedleft\fi\fi\sf\scriptsize \red{RMK: #1}}}

\newcommand\red[1]{\textcolor{red}{#1}}

\newcommand{\be}{\begin{equation}}
\newcommand{\ben}{\begin{equation}\nonumber}
\newcommand{\ee}{\end{equation}}
\newcommand{\bp}{\begin{para}}
\newcommand{\ep}{\end{para}}

\newcommand{\CC}{\mathbb C}

\newcommand{\fid}{\mathrm{fid}}

\newcommand{\supp}{\mathrm{supp}\,}

\newcommand{\appr}{\mathrm{app}}

\def\im{\textrm{im}\,}

\def\D{\mathbb{D}}

\newsavebox{\dotbox}

\newtheorem{proposition}{\textbf{Proposition}}
\newtheorem{lemma}[proposition]{\textbf{Lemma}}
\newtheorem{corollary}[proposition]{\textbf{Corollary}}
\newtheorem{theorem}[proposition]{\textbf{Theorem}}

\theoremstyle{definition}
\newtheorem{definition}{\textbf{Definition}}

\newtheorem*{example*}{\textbf{Example}}
\newtheorem*{remark}{\textbf{Remark}}
\theoremstyle{remark}      
\newtheorem*{rem*}{Remark}

\newtheorem{claim}{Claim}

\newcounter{para}[section]
\newenvironment{para}[2][]{\refstepcounter{para}\noindent\ignorespaces{\bf #1\thepara. #2.} \rmfamily}{\noindent\ignorespacesafterend\bigskip}

\numberwithin{proposition}{section}
\numberwithin{definition}{section}
\begin{document}
\title[Sectional curvature asymptotics]{Sectional curvature asymptotics of the Higgs bundle moduli space}
\date{\today}

\author{Jan Swoboda}
\address{Mathematisches Institut der Universit\"at M\"unchen\\Theresienstra{\ss}e 39\\D--80333 M\"unchen\\ Germany}
\email{swoboda@math.lmu.de}


\maketitle

\begin{abstract}
We determine the asymptotic behavior in the limit of large Higgs fields of the sectional curvatures of the natural $L^2$ hyperk\"ahler metric $G_{L^2}$ of the moduli space $\mathcal M$ of rank-$2$ Higgs bundles on a Riemann surface $\Sigma$ away from the discriminant locus. It is shown that their leading order part is given by a sum of Dirac type contributions on $\Sigma$, for which we find explicit expressions.
\end{abstract}

\section{Introduction}
%
%

Starting with Hitchin's seminal article \cite{hi87}, the moduli space $\mathcal M(r,d)$ of solutions to the self-duality equations on a vector bundle $E$ of rank $r$ and degree $d$ over a Riemann surface $\Sigma$  has been the object of intense research from a number of quite different perspectives. By work of Donaldson \cite{Do87} (which has been extended from Riemann surfaces to higher dimensional K\"ahler manifolds by Corlette \cite{Co88}), the set of irreducible representations of the fundamental group of the Riemann surface into the Lie group $\operatorname{SL}(r,\C)$ is parametrized by (gauge equivalence classes of)  solutions to the  self-duality equations. This parametrization is in terms  of certain equivariant harmonic maps and thus provides an intricate link between the fields of   differential geometry, low-dimensional topology and geometric analysis. Moreover, as an instance of the Kobayashi-Hitchin correspondence, the  moduli spaces of stable Higgs  bundles over a curve is in bijective  correspondence to that of irreducible solutions to the self-duality equations, furnishing a further  link to complex geometry.\\
\medskip\\
In this work, we focus on the differential geometric apects of the moduli space and study asymptotic curvature properties of the natural $L^2$ (or Weil-Petersson type) metric   $G_{L^2}$ it comes equipped with. This metric arises from an (infinite-dimensional) hyperk\"ahler reduction, thus is itself a hyperk\"ahler  metric which by a result due to Hitchin \cite{hi87} is complete if $r$ and $d$ are coprime. The moduli space $\mathcal M(r,d)$ is noncompact, and    its large scale properties remained unexplored until very recently. First steps towards a better understanding of the ends structure of $\mathcal M(2,d)$, both of analytic properties of solutions of the self-duality equations and of the asymptotic structure of $G_{L^2}$, were taken in \cite{msww14,msww15,msww17}. Very recent results concerning the large scale structure of solutions in the case of rank $r\geq3$ are due to Mochizuki \cite{mo16} and Fredrickson \cite{fr16}. The former work also covers Higgs bundles whose determinants are holomorphic quadratic differentials with multiple zeroes, a situation not considered in \cite{msww14}.\\   
\medskip\\
The aim of this article is to carry the study of the large scale geometric properties of the moduli space further and to shed some light on the asymptotic behavior of the sectional curvatures of $G_{L^2}$ in the limit of large Higgs fields. We place our results in the setup considered in \cite{msww17} and again  restrict our attention to the rank-$2$ case and to a sector of $\mathcal M:=\mathcal M(2,0)$ outside the so-called \emph{discriminant locus}. We thus consider the open and dense subset $\mathcal B':=\mathcal B\setminus\mathcal D$, where $\mathcal B=H^0(\Sigma, K_{\Sigma}^2)$ is the complex vector space of holomorphic quadratic differentials on $\Sigma$ and $\mathcal D$ is the discriminant locus consisting of holomorphic quadratic differentials with at least one multiple zero. We call points $q\in\mathcal B'$ simple holomorphic quadratic differentials. Then let $\mathcal M'=\det^{-1}(\mathcal B')$ denote the inverse image of $\mathcal B'$ under the \emph{Hitchin fibration} $\det\colon \mathcal M\to\mathcal B$, $[(A,\Phi)]\mapsto\det\Phi$. The Hitchin fibration has a natural interpretation as an algebraic completely integrable system, and in particular the fibre $\det^{-1}(q)$ over any $q\in\mathcal M'$ is a torus of complex dimension $3(\operatorname{genus}(\Sigma)-1)$, cf.\ \cite{hi87} for details. It admits a global section $\det^{-1}\colon\mathcal B\to\mathcal M$, the \emph{Hitchin section}.  In the forthcoming article \cite{msww17}, asymptotic properties of the restriction to $\mathcal M'$ of the Riemannian metric $G_{L^2}$ are studied. It is shown there that the  restricted   metric is asymptotically close to the so-called  ``semi-flat'' hyperk\"ahler metric $G_{sf}$ associated with the integrable system data of the  Hitchin fibration (cf.\ \cite{fr99,gmn10} for generalities on semi-flat hyperk\"ahler metrics). The metric $G_{sf}$ is a cone metric, i.e.,  it is of the form
\begin{equation*}
G_{sf}=dt^2+t^2 G_{B'}+G_{\operatorname{fibre},Y},
\end{equation*}
where $\{t,Y\}$ denotes a suitable polar coordinate system with radial variable $t$ on the base $\mathcal B'$. Here $G_{B'}$ is a Riemannian metric on the unit sphere $\{(t,Y)\in\mathcal B'\mid t=1\}$ and $G_{\operatorname{fibre},Y}$ is an intrinsically  flat metric on the fibre $\det^{-1}(t^2Y)$ over the point $t^2Y\in B'$ which depends on $Y$ but not on $t$. We let  $\mathcal{SH}$ denote the image of $\mathcal M'$ under the Hitchin section.
The goal of this article is to understand the asymptotic   properties of the sectional curvatures of $(\mathcal M',G_{L^2})$ as $t\to\infty$. Since it is shown in \cite{msww17} that $G_{L^2}$ is asymptotically flat along the fibres of the Hitchin fibration, we consider here the sectional curvatures in direction of two-planes    tangent  to $\mathcal{SH}$.
In this respect the results presented in this article complement those  of \cite{msww17}.\\
\medskip\\
To state our main result, we need to introduce the following pieces of notation. Let $q\in\mathcal B'$ be a simple holomorphic quadratic differential. We canonically identify all tangent spaces of $\mathcal B'$ with $H^0(\Sigma,K_{\Sigma}^2)$. For $t>0$ and a linearly independent pair $\dot f_1,\dot f_2\in H^0(\Sigma,K_{\Sigma}^2)$, let $(X_t,Y_t)\in T_{\det^{-1}(t^2q)}\mathcal{SH}$ be its image of $(\dot f_1,\dot f_2)$ under the differential of $\det^{-1}$. For convenience we further assume   that  the pair $(X_{\infty},Y_{\infty})$ is $L^2$ orthonormal. By the convergence $(X_{t},Y_{t})\to (X_{\infty},Y_{\infty})$ as $t\to\infty$ (cf.\ \textsection\ref{subsect:apprtangentspaces}) the pair $(X_{\infty},Y_{\infty})$ forms an approximately orthonormal frame for $t$ sufficiently large.   We are interested in   the leading order asymptotics as $t\to\infty$ of the sectional curvature of the tangent two-plane  $\Pi(X_t,Y_t)$ spanned by $(X_t,Y_t)$.

\begin{theorem}\label{thm:mainthm} 
For a simple holomorphic quadratic differential $q\in H^0(\Sigma,K_{\Sigma}^2)$ with zero set $\mathfrak p=q^{-1}(0)$ and for $t>0$ sufficiently large, let $(A_t,t\Phi_t)\in\mathcal{SH}$ be the image of $t^2q$ under the Hitchin section. Let $(X_t,Y_t)$ be the  pair of linearly independent tangent vectors  induced by   $(\dot f_1,\dot f_2)\in    H^0(\Sigma, K_{\Sigma}^2)$, and let  $\Pi(X_t,Y_t)$ be the  two-plane spanned by $(X_t,Y_t)$.  Then the sectional curvature $K$ of $\Pi(X_t,Y_t)$ with respect to $G_{L^2}$ satisfies
\begin{equation*}
K(\Pi(X_t,Y_t)) = t^{-\frac{4}{3}}\sum_{p\in\mathfrak p} \Lambda(\dot f_1(p),\dot f_2(p),\dot f_1(p),\dot f_2(p))+\mathcal O(t^{-\frac{5}{3}}).
\end{equation*} 
Here $\Lambda$ is some $\R$-multilinear form, which does not depend on $p$, $q$ or $t$.
\end{theorem}

We briefly comment on the method of proof of the theorem. As mentioned above, the manifold $\mathcal M$ with its metric $G_{L^2}$ arises as the hyperk\"ahler quotient from a certain (infinite-dimensional) Banach manifold and hence may be placed into the geometric setup considered by Jost--Peng \cite{JP}. This provides us with a formula for the Riemann curvature tensor, hence the sectional curvatures $K$ of $G_{L^2}$ in terms of the Green operators arising from the deformation complex associated with the self-duality equations as stated in Eq.\ \eqref{sel.dua.equ} below. A further ingredient in the proof is for $t$ sufficiently large the rather explicit parametrization  in terms of  holomorphic quadratic differentials of the image $\mathcal{SH}$ of the Hitchin section  as well as its tangent spaces. Together with a uniform bound on the operator norms of the associated Green operators, which we derive here, it permits us to identify the leading order contribution to $K(\Pi(X_t,Y_t))$ in the limit $t\to\infty$.\\
\medskip\\
This article initiates the study of the large scale curvature properties of $(\mathcal M,G_{L^2})$. Related articles dealing with curvatures of Weil-Petersson type metrics on various moduli spaces include the following. Bielawski's article \cite{Bi08} is devoted to curvature properties  of K\"ahler and hyperk\"ahler quotients and contains upper and lower bounds on the sectional curvatures in terms of various quantities related to the metric.  In the case of  $(\mathcal M, G_{L^2})$  he proves a uniform upper bound on the sectional curvatures at a point $(A,\Phi)$ in terms of certain algebraic quantities associated with $(A,\Phi)$.  Furthermore, Biswas and Schumacher \cite{BS06} consider a related metric on the moduli space of Higgs bundles over a K\"ahler manifold $X$, derive explicit expressions for its curvature tensor, and in the case where $X$ is a Riemann surface   show nonnegativity of the holomorphic sectional curvatures. Both articles are based on a careful evaluation of the relevant Green operator appearing in the work \cite{JP} mentioned above. These and related  methods have found  applications in a number of further instances, such as to the curvatures of the moduli space of self-dual connections on bundles over four-manifolds (\cite{GP,It, JP}), the moduli space of $\Sp(1)$-instantons over the $4$-sphere (\cite{Hab}),  the Teichm\"uller moduli space of surfaces of genus $\gamma\geq2$ (\cite{JP}), and to certain moduli spaces of K\"ahler-Einstein manifolds (\cite{Siu,Schu}), to name a few.

\subsection*{\bf{Acknowledgments}}
It is a pleasure to thank Rafe Mazzeo and Hartmut Wei{\ss} for a number of useful discussions related to this work.

\section{The  Higgs bundle  moduli space}

We introduce the setup and recall the construction and some basic properties of the $L^2$ metric on the Higgs bundle moduli space, following   Hitchin \cite{hi87}. Then we collect the results obtained in \cite{msww14} concerning the ends structure of the moduli space, as far as these are needed later on.
%
\subsection{Higgs bundles and the self-duality equations}
Let $\Sigma$ be a compact Riemann surface of genus $\gamma\geq2$ and $E \to \Sigma$ a complex vector bundle of rank $2$. We denote by $\End(E)$ and $\mathfrak{sl}(E)$ the   bundles of endomorphisms, respectively tracefree endomorphisms of $E$.   For a hermitian metric $h$ on $E$, we let $\mathfrak{su}(E)=\mathfrak{su}(E,h)$ denote the subbundle of endomorphisms which are skew-hermitian with respect to $h$. We use the notation $\mathcal G=\SU(E,h)$ for the group of special unitary gauge transformations, and write $\mathcal G^c=\SL(E)$ for its complexification. Let further $K_{\Sigma}\to \Sigma$ be the canonical line bundle of $\Sigma$. The choice of a holomorphic structure for $E$ is equivalent with the choice of a {\em Cauchy-Riemann operator} $\delb\colon\Omega^0(E)\to\Omega^{0,1}(E)$, and thus we may consider a holomorphic vector bundle as a pair $(E,\delb)$. A {\em Higgs field} $\Phi$ is a holomorphic section of $\End(E)\otimes K_{\Sigma}$, i.e.\ $\Phi\in H^0(\Sigma,\End(E)\otimes K_{\Sigma})$. By a {\em Higgs bundle} of holomorphically trivial determinant we mean a triple $(E,\delb,\Phi)$ such that $\det E:=\Lambda^2 E$ is {\em holomorphically} trivial and the Higgs field $\Phi$ is traceless, i.e.\ $\Phi\in H^0(\Sigma,\mf{sl}(E)\otimes K_{\Sigma})$.   By a Higgs bundle we shall always mean a Higgs bundle of holomorphically trivial determinant. The   group $\mc G^c$  acts on Higgs bundles diagonally as
\begin{equation*}
g\cdot(E,\bar\partial,\Phi)=(E,g^{-1}\circ\bar\partial\circ g,g^{-1}\Phi g).
\end{equation*}
In order to obtain a smooth   moduli space  we need to restrict to so-called {\em stable Higgs bundles}. In our setting, where the degree of $E$ vanishes, a Higgs bundle $(E,\delb,\Phi)$ is stable if and only if any $\Phi$-invariant holomorphic subline bundle  of $E$, i.e.\ a holomorphic subline bundle $L$ satisfying $\Phi(L)\subset L\otimes K_{\Sigma}$, is of negative degree. We denote by
\begin{equation*}
\mc M=\frac{\{\mbox{stable Higgs bundles}\}}{\mc G^c}
\end{equation*}
the resulting {\em moduli space of stable Higgs bundles}, which can be proven  to be  a smooth complex manifold of dimension $6(\gamma-1)$. It is a non-obvious fact that $\mc M$  carries a {\em natural hyperk\"ahler metric} which naturally appears by reinterpreting the holomorphic data in gauge theoretic terms. To explain this, we fix a hermitian metric $h$ on $E$. Holomorphic structures, as given by a Cauchy-Riemann operator $\delb$ are then in bijective correspondence with special unitary connections, this corresponding being furnished by mapping a unitary connection $d_A$ to its $(0,1)$-part $\delb_A$. After the choice of a base connection the unitary connection is in turn determined by an element in  $\Omega^{0,1}(\Sigma,\mf{sl}(E))$. The above action of the group $\mathcal G^c$ of complex gauge transformations   thus induces one on the set of pairs $(A,\Phi)$. We denote this action by $(A^g,\Phi^g)$ for $g\in\mc G^c$. Hitchin proves that in the $\mc G^c$-equivalence class $[(E,\delb,\Phi)]=[(A,\Phi)]$ there exists a representative $(B,\Psi)=(A^g,\Phi^g)$, unique up to modification by special unitary gauge transformations, such that the so-called {\em self-duality equations}
\begin{equation}\label{sel.dua.equ}
\mc H(B,\Psi):=\begin{pmatrix}F_B^{\perp}+[\Psi\wedge\Psi^*]\\\delb_B\Psi\end{pmatrix}=0
\end{equation}
hold. Here, $F_{B}^{\perp}\in\Omega^2(\Sigma,\frak{su}(E))$ denotes the traceless part of the curvature of the connection $B$ and $\Phi^*\in\Omega^{0,1}(\Sigma,\mathfrak{sl}(E))$ is the hermitian conjugate with respect to $h$. We refer to $\mc H$ as the \emph{nonlinear Hitchin map}. Stability of $(E,\Phi)$ translates into the irreducibility of $(B,\Psi)$. It follows that there is a diffeomorphism
\begin{equation*}
\mc M\cong\frac{\{(B,\Psi) \mid(B,\Psi)\mbox{ solves~\eqref{sel.dua.equ} and is irreducible}\}}{\mc G}.
\end{equation*}
The  self-duality   equations~\eqref{sel.dua.equ} can be interpreted as a {\em hyperk\"ahler moment map} with respect to the natural action of the special unitary gauge group $\mc G$ on the quaternionic vector space $\Omega^{0,1}(\Sigma,\mf{sl}(E))\times\Omega^{1,0}(\Sigma,\mf{sl}(E))$ with its natural $L^2$ metric, cf.\ \cite{hi87,HKLR} for details. Consequently, this metric descends to a hyperk\"ahler metric $G_{L^2}$ on the quotient $\mc M$. We describe this metric next.

%
\subsection{The $L^2$ metric}
In the following, adjoints of differential operators are always understood to be taken with respect to a fixed Riemannian metric compatible with the complex structure of the Riemann surface $\Sigma$. We let $\operatorname{vol}_{\Sigma}$ denote the associated area form. Fix a pair $(A,\Phi)\in\mathcal H^{-1}(0)$ and consider the deformation complex
\begin{multline}\label{eq:deformcomplex}
0\to\Omega^0(\Sigma,\su(E))\stackrel{i_{(A,\Phi)}}{\longrightarrow}\Omega^1(\Sigma,\su(E))\oplus\Omega^{1,0}(\Sigma,\mf{sl}(E))\\
\stackrel{L_{(A,\Phi)}}{\longrightarrow}\Omega^2(\Sigma,\su(E))\oplus\Omega^2(\Sigma,\mf{sl}(E))\to0.
\end{multline}
The first differential is the linearized action of $\mc G$ at $(A,\Phi)$,
\begin{equation*}
i_{(A,\Phi)}(\gamma)=(d_A \gamma,[\Phi\wedge\gamma]),
\end{equation*}
while the second is the linearization of the Hitchin map $\mathcal H$,
\begin{equation}\label{eq:linearizedhitchinmap}
L_{(A,\Phi)}(\dot A,\dot\Phi)=\begin{pmatrix}d_A\dot A+[\dot\Phi\wedge\Phi^*]+[\Phi\wedge\dot\Phi^*]\\
\delb_A\dot\Phi+[\dot A\wedge \Phi]\end{pmatrix}.
\end{equation}
Exactness of the above sequence holds whenever the solution $(A,\Phi)$ is irreducible, cf.\ the  discussion in~\cite{hi87}, which will always be the case here.
The tangent space of $\mc M$ at $[(A,\Phi)]$ can  therefore be identified with the quotient 
\begin{equation*}
\ker L_{(A,\Phi)}/\im i_{(A,\Phi)}\cong\ker L_{(A,\Phi)}\cap(\im i_{(A,\Phi)})^\perp.
\end{equation*}
Now for a  vector $(\dot A, \dot \Phi)$,
\begin{equation}\label{eq:Cgaugecondition}
(\dot A,\dot\Phi)\perp\im i_{(A,\Phi)}\qquad\Longleftrightarrow\qquad d_A^*\dot A+\Re[\Phi^*\wedge\dot\Phi]=0.
\end{equation}
The real part enters since we consider the Riemannian metric on $\Omega^1(\Sigma,\mf{su}(E))$ coming from the hermitian metric on $\Omega^{1,0}(\Sigma,\mf{sl}(E))$. If this condition is satisfied  we will say that $(\dot A, \dot \Phi)$ is in {\em Coulomb gauge}. For tangent vectors $(\dot A_i, \dot \Phi_i)$, $i=1,2$, in Coulomb gauge the induced $L^2$ metric is given by
\begin{equation*}
G_{L^2}((\dot A_1,\dot\Phi_1),(\dot A_2,\dot\Phi_2))=\int_{\Sigma}\langle\dot A_1,\dot A_2\rangle+\langle\dot\Phi_1,\dot\Phi_2\rangle\,\operatorname{vol}_{\Sigma}.
\end{equation*}
Here we identify $\Omega^1(\Sigma,\mf{su}(E))\cong\Omega^{0,1}(\Sigma,\mf{sl}(E))$ and use the hermitian inner product on $\mathfrak{sl}(2,\C)$ given by $\langle A,B\rangle=\Tr(AB^{\ast})$.

\subsection{The structure of ends of the Higgs bundle moduli space}

From now on it is always understood that the Higgs fields $\Phi$ we consider is  {\em simple}, in the sense that the holomorphic quadratic differential $\det\Phi$ has only simple zeroes. This implies, in particular, the stability of any Higgs pair $(A,\Phi)$ in the above sense. The set of (gauge equivalence classes of) Higgs pairs with simple Higgs field is an open and dense subset of $\mathcal M$. It admits a compactification by so-called  `limiting configurations', consisting of pairs $(A_\infty, \Phi_\infty)$ 
which satisfy a decoupled version of the self-duality equations \eqref{sel.dua.equ}, namely
\begin{equation*}
F_{A_\infty}^{\perp}=0,\quad[\Phi_\infty\wedge\Phi^*_\infty]=0,\quad\delb_{A_\infty}\Phi_\infty=0.
\end{equation*}
Each $A_\infty$ is   flat with simple poles in the set $\mathfrak p$ of zeroes of $\det\Phi$, while the limiting Higgs fields are holomorphic with respect to these connections and have a specified behavior near these poles. We will return to a description of limiting configurations in more concrete terms in the next section.

\begin{theorem}[\cite{msww14}, existence and deformation theory of limiting configurations]\label{thm:existencelimconf}
Let $(A_0,\Phi_0)$ be a Higgs pair such that $q:= \det \Phi_0$ has only simple zeroes. Then there exists a complex 
gauge transformation $g_\infty$ on $\Sigma^\times=\Sigma\setminus q^{-1}(0)$ which transforms $(A_0,\Phi_0)$ into 
a limiting configuration. Furthermore, the space of limiting configurations with fixed determinant $q \in H^0(\Sigma,K_{\Sigma}^2)$  having simple zeroes is a torus of dimension $6\gamma-6$, where $\gamma$ is the genus of $\Sigma$.
\end{theorem}

Each limiting configuration can be desingularized into  a smooth solution of the self-duality equations. This desingularization gives rise to a parametrization of  a neighborhood of that part of the boundary of  $\mathcal M$ which consisits of simple Higgs fields.

\begin{theorem}[\cite{msww14}, desingularization theorem]
If $(A_\infty,\Phi_\infty)$ is a limiting configuration, then there exists a family $(A_t,\Phi_t)$ of solutions of the rescaled Hitchin equation
\begin{equation*}
F_A^{\perp} +t^2[\Phi\wedge\Phi^*]=0,\quad\delb_A \Phi=0
\end{equation*}
provided $t$ is sufficiently large, where
\begin{equation*}
(A_t,\Phi_t) \longrightarrow (A_\infty,\Phi_\infty)
\end{equation*}
as $t \nearrow \infty$, locally uniformly on $\Sigma^\times$ along  with all derivatives, at an exponential rate in $t$. Furthermore, $(A_t,\Phi_t)$  is complex gauge equivalent to $(A_0,\Phi_0)$ if $(A_\infty,\Phi_\infty)$ is the limiting configuration associated with $(A_0,\Phi_0)$.
\end{theorem}

Near the zeroes of $q=\det\Phi$ the solutions $(A_t,\Phi_t)$ obtained in this theorem can be arranged to be in  standard form, i.e.\ to coincide with the so-called {\emph{fiducial solutions}} $(A_t^{\fid},\Phi_t^{\fid})$, which will be further discussed in the next section.

\section{Approximate local horizontal tangent frames}

We give a rather explicit description of the image $\mathcal{SH}$ of the Hitchin section and its tangent spaces, which further on will permit us to determine asymptotic properties of the associated sectional curvatures in the limit $t\to\infty$. All the results discussed in this section are either contained in \cite{msww14} or will be part of  the forthcoming article \cite{msww17}, which also contains complete proofs.

\subsection{The (approximate)  Hitchin section}\label{subsect:hitchinsection}

Let $\mc B=H^0(\Sigma,K_{\Sigma}^2)$ denote the space of holomorphic quadratic differentials, and $\Lambda \subset \mc B$ the so-called {\emph{discriminant locus}}, consisting of holomorphic quadratic differentials for which at least one zero is not simple. This is a closed subvariety which is invariant under the multiplicative  action of $\CC^\times$, and hence $\mc B':=\mc B\setminus \Lambda$ is an open dense subset of $\mc B$. The determinant is invariant under conjugation and therefore descends to a holomorphic map
\begin{equation*}
\det\colon \mc M\to\mc B,\qquad[(\delb,\Phi)]\mapsto\det\Phi,
\end{equation*}
called the {\em Hitchin fibration}~\cite{hi87}. This map is proper and surjective. Let $\mathcal M'$ denote the inverse image of $\mathcal B'$ under  $\det$. We remark that the space $\mc M'$ with its natural complex symplectic structure admits the interpretation as a  completely integrable system over $\mathcal B'$, cf.\ \cite[Section 44]{gs90} and~\cite{fr99,gmn10,gmn13,ne14}. In particular, the fibres of the Hitchin fibration are  affine tori. Now fix a holomorphic square root $\Theta=K_{\Sigma}^{1/2}$ of the canonical bundle. Then $E$ splits holomorphically as $E=\Theta \oplus \Theta^*$. We define the {\em Hitchin section} by
\begin{equation*}
{\det}^{-1}\colon \mc B\to\mc M,\quad{\det}^{-1}(q)=\left(\delb_{\Theta \oplus \Theta^*} ,\Phi_q=\begin{pmatrix}0&-q\\1&0\end{pmatrix}\right).
\end{equation*}
We let $\mathcal{SH}={\det}^{-1}(\mathcal B')$ denote the image of $\mathcal B'$ under the Hitchin section.\\
\medskip\\
We need to introduce some more pieces of notation. For a simple holomorphic quadratic differential $q\in\mathcal B'$ we denote by $\mathfrak p=\mathfrak p(q)=q^{-1}(0)\subset\Sigma$ its set of zeroes. Around each $p\in\mathfrak p$ we choose a local complex coordinate $z$ centered at $p$ and such that the unit disks $\D(p)$, $p\in\mathfrak p$, are pairwise disjoint. We further set $\Sigma^{\operatorname{ext}}:=\Sigma\setminus\bigcup_{p\in\mathfrak p}\D(p)$ and $\Sigma^{\times}:=\Sigma\setminus\mathfrak p$.\\
\noindent\\
Theorem \ref{thm:existencelimconf} yields  for a simple holomorphic quadratic differential $q$ a (singular) limiting configuration $(A_{\infty},\Phi_{\infty})$ such that $\det\Phi_{\infty}=q$.   It allows for a desingularization, to be described in a moment, by a family $S_t^{\appr}:=(A_t^{\appr},t\Phi_t^{\appr})$ of smooth maps in such a way that $S_t^{\appr}$ lies exponentially close to   the image of $t^2q$ under the Hitchin section. The assignment $t  \mapsto S_t^{\appr} $ is therefore a good approximation of the Hitchin section  along the ray $t\mapsto t^2q$. Explicitly, with respect to the decomposition $E=\Theta \oplus \Theta^*$ and a fixed hermitian metric $H$ with Chern connection $A_H$ on $E$,  the components of $S_t^{\appr}$ are given by
\begin{equation}\label{eq:atappr}
A_t^{\appr}=A_t^{\appr}(q)=A_{H} + 4 f_t(|q|)\Im \bar \partial \log |q| \otimes \begin{pmatrix}
 i & 0 \\ 0 & -i
\end{pmatrix}
\end{equation}
and
\begin{equation}\label{eq:phitappr}
\Phi_t^{\appr} = \Phi_t^{\appr}(q)=\begin{pmatrix} 
 	0 & |q|^{-1/2} e^{-h_t(|q|)}q\\
 	|q|^{1/2}e^{h_t(|q|)} & 0
 \end{pmatrix}. 
 \end{equation}
By construction, $\det\Phi_t^{\appr}=q$ for all $t$. On the disks $\D(p)$, where $p\in\mathfrak p$,  the approximate solution $S_t^{\appr}$ coincides with the so-called \emph{fiducial solution} $(A_t^{\fid},t\Phi_t^{\fid})$. These are a family of exact solutions to the self-duality equations, which are defined on the whole complex plane $\C$, and satisfy $\det    \Phi_t^{\fid}=-z\, dz^2$. Following  \cite[\textsection 3.2]{msww14}, we explain their properties as far as these will be needed later on.  We set
 \begin{equation*}
  f_t:=\frac{1}{8}+\frac{1}{4}rh_t'\colon[0,\infty)\to\R,
\end{equation*}
where  the function $h_t\colon[0,\infty)\to\R$ is defined to be the solution of the ODE 
\begin{equation*}
(r\partial_r)^2 h = 8  t^2 r^3 \sinh (2h)
\label{painleve}
\end{equation*}
with specific asymptotic properties as $r\searrow0$ and $r\nearrow\infty$. By the substitution $\rho = \frac{8}{3} t r^{3/2}$ we get $r\partial_r = \frac{3}{2} \rho \partial_\rho$, and writing $h_t(r) = \psi(\rho)$ for some function $\psi$, we obtain the  $t$-independent ODE
\begin{equation*}
(\rho \partial_\rho)^2 \psi = \frac{1}{2} \rho^2 \sinh (2\psi).
\label{sceqn}
\end{equation*}
The equation \eqref{sceqn} is of Painlev\'e III type. It admits a unique solution which decays exponentially and has a the correct behavior as $\rho \to 0$, namely
\begin{equation*}
\begin{array}{rl}
\bullet\ & \psi(\rho) \sim -\log (\rho^{1/3} \left( \sum_{j=0}^\infty a_j \rho^{4j/3}\right), \quad \rho \searrow 0; \\[0.5ex]
\bullet\ & \psi(\rho) \sim K_0(\rho) \sim \rho^{-1/2} e^{-\rho}, \quad \rho \nearrow \infty; \\[0.5ex]
\bullet\ & \psi(\rho)\mbox{ is monotonically decreasing (and hence strictly positive)}.
\end{array}
\label{proph}
\end{equation*}
The notation $A \sim B$ indicates a complete asymptotic expansion. In the first case, for example, for each $N \in \mathbb N$, 
\[
\left|\rho^{-1/3} e^{-\psi(\rho)} - \sum_{j=0}^N a_j \rho^{4j/3}\right| \leq C \rho^{4(N+1)/3},
\]
with a corresponding expansion for any derivative. The function $K_0(\rho)$ is the Macdonald function (or Bessel function 
of imaginary argument) of order $0$; it has a complete asymptotic expansion involving terms of the
form $e^{-\rho} \rho^{-1/2-j}$, $j \geq 0$, as $\rho \to \infty$.  We recall the asymptotic properties of the functions $h_t$ and $f_t$ as obtained in \cite[Lemma 3.4]{msww14}.

\begin{lemma}\label{f_t-h_t-function}
The functions $f_t(r)$ and $h_t(r)$ have the following properties: 
\begin{enumerate}[(i)]
	\item As a function of  $r$, $f_t$ has a double zero at $r=0$ and increases monotonically from $f_t(0) = 0$ to the limiting value $1/8$ as $r \nearrow \infty$.  In particular, $0 \leq f_t \leq \frac 18$.
	\item As a function of $t$, $f_{t}$ is also monotone increasing. Further, $\lim_{t \nearrow \infty} f_t = f_\infty \equiv \frac18$ 
uniformly in $\mc C^\infty$ on any half-line $[r_0,\infty)$, for $r_0 > 0$. 
	\item There are uniform estimates 
\[
\sup_{r >0}  r^{-1} f_t(r) \leq C t^{2/3} \quad \text{and}\quad \sup_{r >0} r^{-2} f_t(r) \leq C t^{4/3},
\]
where $C$ is independent of $t$. 
	\item When $t$ is fixed and $r \searrow 0$, then $h_t(r) \sim -\tfrac{1}{2} \log r + b_0 + \ldots$, where $b_0$ is an explicit constant. Moreover, $|h_t(r)| \leq C \exp( -\tfrac83 t r^{3/2})/ ( t r^{3/2})^{1/2}$ uniformly for $t \geq t_0 > 0$, $r \geq r_0 > 0$.    
	\item There is a uniform estimate  
	\[
	\sup_{r \in(0,1)} r^{1/2} e^{\pm h_t(r)} \leq C,  	
	\]
	where $C$ is independent of $t$.
\end{enumerate}
\end{lemma}

The results in \cite{msww14} show that the approximate solutions  $S_t^{\appr}$  satisfy the self-duality equations up to some error, which decays (in any $C^k$ norm) at an exponential rate to $0$  as $t\to\infty$. Furthermore, for $t$ sufficiently large, it is proven that the complex gauge orbit of $S_t^{\appr}$ contains a unique nearby exact solution $(A_t,t\Phi_t)$. The set of all  approximate solutions  obtained in this way forms a smooth Banach manifold which we denote by $\mathcal M^{\appr}$, and we may  think of its quotient by $\mathcal G$ as a good approximation of the submanifold $\mathcal{SH}$ (the  image of the Hitchin section) of $\mathcal M'$. Our next aim  is to construct a     frame of tangent vectors along $\mathcal{SH}$ in sufficiently explicit terms, so that it is well-suited for the subsequent  sectional curvature computations.

\subsection{Approximate solutions and approximate tangent spaces}\label{subsect:apprtangentspaces}

The   construction of a local frame is carried out first locally on the disks $\D(p)$ about the set $\mathfrak p$ of zeroes of the holomorphic quadratic differential $q=\det \Phi_t^{\appr}$. 
With respect to the standard complex coordinate $z$ on $\D(p)$ we write   $q=f dz^2=-z\, dz^2$.   We consider the variation of the approximate solution $S_t^{\appr}$ with respect to $q$, i.e.~we pick $\dot q \in T_{t^2q}H^0(\Sigma,K_{\Sigma}^2)\cong  H^0(\Sigma,K_{\Sigma}^2)$, where restricted to   $\D(p)$,    $\dot q= \dot f dz^2$ for some holomorphic function $\dot f$. Then  we set  $(\dot{A}_t^{\appr},t\dot{\Phi}_t^{\appr})$, where 
\begin{equation*}
\dot{\Phi}_t^{\appr}:=\left.\frac{d}{d\varepsilon}\right|_{\varepsilon=0}\Phi_t^{\appr}(q+\varepsilon \dot q)\qquad\textrm{and}\qquad \dot{A}_t^{\appr}:=\left.\frac{d}{d\varepsilon}\right|_{\varepsilon=0}A_t^{\appr}(q+\varepsilon \dot q).
\end{equation*}
Thus $(\dot{A}_t^{\appr},t\dot{\Phi}_t^{\appr})$ is a tangent vector to $\mathcal M^{\appr}$ at the point  $S_t^{\appr}$. Using the expressions in \eqref{eq:atappr} and \eqref{eq:phitappr}, we obtain that
\begin{multline}\label{eq:dotPhiapp}
\dot{\Phi}_t^{\appr}=\\
\begin{pmatrix}
 0 & e^{-h_t(|f|)}  |f|^{-\frac{1}{2}} \Big(\dot f-(\frac{1}{2}+|f|h_t'(|f|))f\Re\frac{\dot f}{f}\Big)  \\
  e^{h_t(|f|)}|f|^{1/2}  \Re\frac{\dot f}{f} \Big( \frac{1}{2} + |f| h_t'(|f|)\Big) & 0	
 \end{pmatrix} dz
\end{multline}
and 
\begin{equation}\label{eq:dotAapp}
\dot A_t^{\appr} =\left( -2   f_t'(|f|) |f| \Re\frac{\dot f}{f} \Im \frac{df}{f} -2 f_t(|f|) d \Im\frac{\dot f}{f}\right)  \begin{pmatrix}
 i & 0 \\ 0 & -i
 \end{pmatrix}. 
\end{equation}
The pair $(\dot{A}_t^{\appr},t\dot{\Phi}_t^{\appr})$  will in general  not satisfy the Coulomb gauge   condition. In   order to meet it in good approximation we let
\begin{equation*}
\gamma_t:=-\Big( \frac 14+\frac{1}{2}|f|h_t'(|f|)\Big) \Im\frac{\dot f}{f}\begin{pmatrix}i&0\\0&-i\end{pmatrix}=-2f_t(|f|)\Im\frac{\dot f}{f}\begin{pmatrix}i&0\\0&-i\end{pmatrix}
\end{equation*}
and define for $ \dot \Phi_t^{\appr}$ as  in \eqref{eq:dotPhiapp} and  $ \dot A_t^{\appr}$ as in \eqref{eq:dotAapp} 
\begin{multline*}
(\alpha_t,t\varphi_t):=  (\dot A_t^{\appr},  t \dot \Phi_t^{\appr})-i_{S_t^{\appr}}\gamma_t\\
=( \dot A_t^{\appr} -d_{A_t^{\appr}}\gamma_t, t\dot \Phi_t^{\appr}-t[\Phi_t^{\appr}\wedge\gamma_t]).
\end{multline*}
Hence $(\alpha_t,t\varphi_t)$ is a further tangent vector to $\mathcal M^{\appr}$ at $S_t^{\appr}$ but, as one may check, satisfies the Coulomb gauge  condition up to  a much smaller error.  Explicit expressions for $\alpha_t$ and $\varphi_t$ are straightforward to derive. Namely, we obtain that
\begin{multline}\label{eq:correctedhortangentphi}
\varphi_t =\\
\begin{pmatrix} 0&\Big( \frac{1}{2}-|f|h_t'(|f|)\Big)e^{-h_t(|f|)}|f|^{-\frac{1}{2}}\dot f\\\Big( \frac{1}{2}+|f|h_t'(|f|)\Big)e^{h_t(|f|)}|f|^{\frac{1}{2}}\cdot \frac{\dot f}{f}&0   \end{pmatrix}dz
\end{multline}
and
\begin{equation}\label{eq:correctedhortangalpha}
\alpha_t=f_t'(|f|)|f|^{-1} (\dot f\bar\partial\bar f-\bar{\dot f}\partial f)\begin{pmatrix}1&0\\0&-1\end{pmatrix}.
\end{equation}
We note the pointwise convergence $(\alpha_t,\varphi_t)\to(0,\varphi_{\infty})$ as $t\to\infty$ as follows from Lemma \ref{f_t-h_t-function}, where $\varphi_{\infty}$ is as in \eqref{eq:correctedhortangentphi} with $e^{-h_t(|f|)}$ replaced by $1$.
We may extend  $(\alpha_t,\varphi_t)$  smoothly   over $\bigcup_{p\in\mathfrak p}\D(p)$ such that on $\Sigma^{\operatorname{ext}}$ it agrees with  $(0,\varphi_{\infty})$.   To put $(\alpha_t,t\varphi_t)$ into Coulomb gauge, we need to apply a final gauge correction step. The terms obtained in this step are less explicit, but still  admit  estimates which  turn out to be sufficient for our purposes.

\subsubsection*{{\bf Gauge correction}}\label{subsect:gaugecorrect}
We start with   a short digression on how this final gauge correction step is carried out. Recall from \eqref{eq:Cgaugecondition} the definition of the Coulomb gauge condition. It is standard to show that by addition of an appropriate gauge correction term, this condition can always be satisfied.

\begin{lemma}[Coulomb gauge fixing] For each $(\dot A, \dot \Phi)$ there exists a unique  $\xi \in \Omega^0(\Sigma,\su(E))$ such that $(\dot A, \dot \Phi)-i_{(A,\Phi)}\xi=(\dot A - d_A \xi, \dot \Phi - [\Phi \wedge \xi])$ is in Coulomb gauge. It is given as the unique solution of the linear elliptic equation
\begin{equation}\label{eq:gaugefixeq} 
\Delta_A\xi+\Re[\Phi\wedge[\Phi^{\ast}\wedge\xi]]=d_A^{\ast}\dot A+\Re[\Phi^{\ast}\wedge\dot \Phi].
\end{equation}
\end{lemma}

We carry out the remaining  gauge correction  step for the normalized tangent vector $(\dot A,\dot\Phi)=(t^{-1}\alpha_t,\varphi_t)$, the normalization being chosen in order to obtain uniformly bounded $L^2$ norms. We hence need to   estimate the solution $\xi_t$ of Eq.~\eqref{eq:gaugefixeq} in the case where $(A,\Phi)=S_t^{\appr}=(A_t^{\appr},t\Phi_t^{\appr})$ is an approximate solution as in \eqref{eq:atappr} and \eqref{eq:phitappr}. We abbreviate the right-hand side of this equation as
\begin{equation*}\label{eq:errorexpression}
E_t:=t^{-1} d_{A_t^{\appr}}^{\ast}\alpha_t + t\Re[(\Phi_t^{\appr})^* \wedge \varphi_t].
\end{equation*}
We let $\xi_t$ denote its unique solution. The resulting   tangent vector to $S_t^{\appr}$ in Coulomb gauge therefore is
\begin{equation*}
X_t:=X_t(\dot q):= (t^{-1}\alpha_t,\varphi_t)- i_t\xi_t.
\end{equation*}
We note that for $t=\infty$ the gauge fixing term vanishes, and hence $X_{\infty}$ equals the above vector $(0,\varphi_{\infty}$).

\begin{definition}\label{def:inducedtangentvector}
We call $X_t(\dot q)$ the  tangent vector at $S_t^{\operatorname{app}}$  induced by the holomorphic quadratic differential $\dot q$.
\end{definition}

It remains to consider  the convergence properties of the family $t\mapsto X_t(\dot q)$ in the limit $t\to\infty$ and hence to derive uniform estimates on the gauge correction term $i_t \xi_t$. These are based on the following proposition.

\begin{proposition}\label{prop:errorest}
There is a constant $C$ such that  $E_t$ satisfies the uniform estimate
 
\begin{equation*}\label{eq:c0bounderror}
\|E_t\|_{C^0(\Sigma)}\leq Ct 
\end{equation*}
for all $t\geq1$.
\end{proposition}

The proof of this proposition is given in \cite{msww17}. Of relevance in the following will further  be that 
\begin{equation*}\label{eq:errorexterior}
E_t=\mathcal O(e^{-\beta t})\quad\textrm{on}\quad \Sigma^{\operatorname{ext}}
\end{equation*}
for some constant $\beta>0$, and that  restricted to each disk $\D(p)$, $p\in\mathfrak p$, 
\begin{equation*}
E_t=m_t\begin{pmatrix}1&0\\0&-1\end{pmatrix}
\end{equation*}
for some complex-valued function $m_t$ which we do not need to specify here. It is then shown in \cite{msww17} that the corresponding gauge correction terms $i_t\xi_t$ decay at   rate $t^{-\frac{1}{3}}$ to $0$ in $L^2$ as $t\to \infty$, with a more refined estimate on   $\xi_t$ to be given in Lemma \ref{lem:refinedestgaugecorr}. One thus arrives at the following result.

\begin{lemma}\label{thm:convhorizontal}
There is  a constant $C>0$ which does not depend on $t$ such that
\begin{equation}\label{eq:convhorizontal}
\|X_t-(0,\varphi_{\infty})\|_{L^2(\Sigma)}\leq Ct^{-\frac{1}{3}}
\end{equation}
for all $t\geq1$.
\end{lemma}

\section{Sectional curvatures of the $L^2$ metric}

\subsection{Generalities} 

The study of the curvature properties  of the Higgs bundle moduli space  fits into the following more general setup considered by Jost and Peng in \cite{JP}. Let  $\mathcal M$ be a smooth Banach manifold, endowed with a smooth Riemannian metric $G$, and acted on isometrically by   a Banach Lie group $\mathcal G$ with Banach Lie algebra $\mathfrak g$. Let  $V$ be a Banach space  and   $\phi\colon\mathcal M\to V$ be a  smooth map, and suppose that  the level set $\phi^{-1}(0)$ is invariant under the action of $\mathcal G$. Then the smooth part of the quotient $\mathcal M_0:=\phi^{-1}(0)/\mathcal G$ is again a Banach manifold.  It  inherits from $(\mathcal M,G)$ a Riemannian metric $G_0$ in such a way that the canonical projection $\pi\colon\phi^{-1}(0)\to\mathcal M_0$ is a Riemannian submersion. The main objective of  \cite{JP} is to derive a formula for the sectional curvatures of $G_0$ in terms of that of $G$, the derivatives of $\phi$ and the action by $\mathcal G$. We make the assumption (which is satisfied in the situation to be considered below)  that for every $p\in \phi^{-1}(0)$ the sequence of maps
\begin{equation*}
0\to \mathfrak g\stackrel{i_p}{\longrightarrow} T_p\mathcal M \stackrel{j_p}{\longrightarrow} V\to 0
\end{equation*}
is exact, where (with $t\mapsto g_t$ denoting the one-parameter subgroup generated by $X\in\mathfrak g$)
\begin{equation*}
i\colon \mathfrak g\to \Omega^0(T\mathcal M),\quad i_p(X)=\left.\frac{d}{dt}\right|_{t=0} g_t(p)
\end{equation*}
and
\begin{equation*}
j:= d\phi\colon T\mathcal M\to V.
\end{equation*}
We hence may uniquely identify a tangent vector at the point $[p]\in\mathcal M_0$   with some  $\alpha\in\ker i_p^{\ast}\cap \ker j_p$. Here the adjoint is taken with respect to some fixed inner product on $\mathfrak g$. We denote by $G_p^0$ and $G_p^2$ the inverses of the associated Laplacians $i_p^ {\ast}i_p$ and $j_pj_p^{\ast}$, which exist by exactness of the above sequence. Furthermore, for $\gamma\in T_p\mathcal M$ we define
\begin{equation*}
P_{\gamma}:=(d_{\gamma}i)(p)\colon \mathfrak g\to T_p\mathcal M
\end{equation*}
and
\begin{equation*}
Q_{\gamma}:=(d_{\gamma}j)(p)\colon  T_p\mathcal M\to V.
\end{equation*}
After identifying $\mathfrak g^{\ast}$ with $\mathfrak g$ via the above chosen inner product, the adjoint of the map $P_{\gamma}$ is the map $P_{\gamma}^{\ast}\colon T_p\mathcal M\to\mathfrak g$. The following theorem relates  $\langle R(\alpha,\beta)\gamma,\delta\rangle_{G_0}$, the curvature tensor  of $(\mathcal M_0,G_0)$ evaluated on the tuple $(\alpha,\beta,\gamma,\delta)$ of  tangent vectors  at $[p]$,   to that of $(\mathcal M,G)$.

\begin{theorem}(\cite[Theorem 2.3]{JP}).\label{thm:seccurv}
The Riemann curvature tensor  of $\mathcal M_0$ with respect to the induced metric $G_0$ is
\begin{multline*}
\langle R(\alpha,\beta)\gamma,\delta\rangle_{G_0} =\langle R^{\mathcal M}(\alpha,\beta)\gamma,\delta\rangle_G+\langle G_p^0P_{\alpha}^{\ast}\delta, P_{\beta}^{\ast}\gamma\rangle_{\mathfrak g}-\langle G_p^0P_{\beta}^{\ast}\delta, P_{\alpha}^{\ast}\gamma\rangle_{\mathfrak g} \\
+2\langle G_p^0P_{\delta}^{\ast}\gamma, P_{\alpha}^{\ast}\beta\rangle_{\mathfrak g}+ \langle G_p^2Q_{\beta}\gamma, Q_{\alpha}\delta\rangle_V -  \langle G_p^2Q_{\alpha}\gamma, Q_{\beta}\delta\rangle_V, 
\end{multline*}
where $ R^{\mathcal M}$ is the curvature tensor of $(\mathcal M,G)$.
\end{theorem}

The terms involving $G_p^0$ have a natural interpretation as O'Neill-type contributions to the curvature coming from  the Riemannian submersion $\phi^{-1}(0)\to \mathcal M_0$, while the terms involving $G_p^2$ represent Gau{\ss}-type contributions associated with the embedding $\phi^{-1}(0)\hookrightarrow \mathcal M$.
\medskip\\ 
Our goal is to analyze the sectional curvatures of $(\mathcal M,G_{L^2})$ by employing the curvature formula of Theorem \ref{thm:seccurv}. In our setup, to be described next, the ambient metric $G$ is flat. Thus the sectional curvature of $G_0$ in direction of the $2$-plane spanned by the  orthogonal  frame $\{X ,Y \}$ is  
\begin{multline}\label{eq:sectionalcurvature}
K(X,Y)= \langle R(X,Y)Y,X\rangle_{G_0}\\
=  3\langle G_0P_X^{\ast}Y,P_X^{\ast}Y\rangle_{\mathfrak g}+\langle G_2Q_XX,Q_YY\rangle_{V}-\langle G_2Q_XY,Q_XY\rangle_{V},
\end{multline}
using that (as shown in \cite{JP}) the map $Q$ is symmetric and $P^{\ast}$ is skew-symmetric.

\subsection{The elliptic complex and the associated Laplacians}\label{subsect:Laplacians} 
 
Coming back to the moduli space $\mathcal M$ of solutions to the self-duality equations, the role of the map $\phi$ above is taken by the nonlinear Hitchin map
\begin{align*}\label{eq:nonlinHitchin}
\mathcal H\colon\mathcal A(E,h)\times \Omega^{1,0}(\Sigma,\mathfrak{sl}(E))\mapsto\Omega^2(\Sigma,\mathfrak{su}(E))\oplus\Omega^2(\Sigma,\mathfrak{sl}(E)),\\
  \mathcal H(A,t\Phi):=\begin{pmatrix}F_A^{\perp}+t^2[\Phi\wedge\Phi^{\ast}]\\\bar\partial_A\Phi\end{pmatrix}.
\end{align*}
Its linearization at $(A,t\Phi)$ is the map $L_{(A,t\Phi)}$  as in \eqref{eq:linearizedhitchinmap}, which we write in the form
\begin{eqnarray*}
L_{(A,t\Phi)}=\begin{pmatrix} d_A&tR_{\Phi}\\t S_{\Phi}& \bar \partial_A\end{pmatrix}.
\end{eqnarray*}
Here we denote
\begin{eqnarray*}
R_{\Phi}\colon\Omega^{1,0}(\Sigma,\mathfrak{sl}(E))\to\Omega^2(\Sigma,\mathfrak{su}(E)),\qquad R_{\Phi}\varphi=[\Phi\wedge\varphi^{\ast}]+[\Phi^{\ast}\wedge\varphi]
\end{eqnarray*}
and
\begin{eqnarray*}
S_{\Phi}\colon\Omega^1(\Sigma,\mathfrak{su}(E))\to\Omega^2(\Sigma,\mathfrak{sl}(E)),\qquad S_{\Phi}\alpha=[ \Phi\wedge \alpha^{0,1}].
\end{eqnarray*}
Recall that the infinitesimal action of the group of unitary gauge transformations at $(A,t\Phi)$ is given by  
\begin{align*}
i_{(A,t\Phi)}\colon\Omega^0(\Sigma,\mathfrak{su}(E))\to\Omega^1(\Sigma, \mathfrak{su}(E))\oplus\Omega^{1,0}(\Sigma, \mathfrak{sl}(E)),\\
i_{(A,t\Phi)}\gamma=(d_A\gamma,[t\Phi\wedge\gamma]).
\end{align*}
The maps $i_{(A,t\Phi)}$  and $L_{(A,t\Phi)}$  combine to give an elliptic complex as in \eqref{eq:deformcomplex}.
The induced Laplace operators in degree $0$ and $2$ are
\begin{equation}\label{eq:laplace0}
D_{(A,t\Phi)}^0:= i_{(A,t\Phi)}^{\ast}i_{(A,t\Phi)}\colon\Omega^0(\Sigma,\mathfrak{su}(E))\to\Omega^0(\Sigma,\mathfrak{su}(E))
\end{equation}
and
\begin{multline}\label{eq:laplace2}
D_{(A,t\Phi)}^2:= L_{(A,t\Phi)}L_{(A,t\Phi)}^{\ast}\colon\\
\Omega^2(\Sigma,\mathfrak{su}(E)) \oplus\Omega^2(\Sigma,\mathfrak{sl}(E))\to\Omega^2(\Sigma,\mathfrak{su}(E)) \oplus\Omega^2(\Sigma,\mathfrak{sl}(E)).
\end{multline}
Finally, the above operators $P$ and $Q$  read in the present context 
\begin{multline*}
P_{(\alpha,\varphi)}\colon \Omega^0(\Sigma,\mathfrak{su}(E))\to \Omega^1(\Sigma,\mathfrak{su}(E))\oplus \Omega^{1,0}(\Sigma,\mathfrak{sl}(E)),\\
P_{(\alpha,\varphi)}\gamma=\left.\tfrac{d}{dt}\right|_{\varepsilon=0}i_{(A+\varepsilon\alpha,t\Phi+\varepsilon\varphi)}=([\alpha\wedge\gamma],[\varphi\wedge\gamma])
\end{multline*}
and
\begin{multline*}
\label{eq:Q}
Q_{(\alpha,\varphi)}\colon \Omega^1(\Sigma,\mathfrak{su}(E))\oplus \Omega^{1,0}(\Sigma,\mathfrak{sl}(E))\to \Omega^2(\Sigma,\mathfrak{su}(E))\oplus \Omega^2(\Sigma,\mathfrak{sl}(E)),\\
 Q_{(\alpha,\varphi)}(\beta,\psi)=\left.\tfrac{d}{d\varepsilon}\right|_{\varepsilon=0}L_{(A+\varepsilon\alpha,t\Phi+\varepsilon\varphi)}(\beta,\psi)\\
 =([\alpha\wedge\beta]+[\varphi\wedge\psi^{\ast}]+[\varphi^{\ast}\wedge\psi],[\alpha^{0,1}\wedge\psi]+[\beta^{0,1}\wedge\varphi]).
\end{multline*}
By computations similar to those  in the next section, the adjoint of   $P_{(\alpha,\varphi)}$ is the operator 
\begin{multline*}
P_{(\alpha,\varphi)}^{\ast}\colon \Omega^1(\Sigma,\mathfrak{su}(E))\oplus \Omega^{1,0}(\Sigma,\mathfrak{sl}(E))\to  \Omega^0(\Sigma,\mathfrak{su}(E)),\\
 P_{(\alpha,\varphi)}^{\ast}(\beta,\psi)= \ast [\ast\alpha\wedge\beta]  +\frac{i}{2}\ast[\varphi\wedge\psi^{\ast}]-\frac{i}{2}\ast[\varphi^{\ast}\wedge\psi].
\end{multline*}

\subsubsection*{{\bf Computation of the adjoint operators}}
We wish to write down the operators $D_{(A,t\Phi)}^0$ and $D_{(A,t\Phi)}^2$ in more explicit terms, and thus need to calculate various adjoints with respect to the hermitian inner product on $\mathfrak{sl}(2,\C)$ given by $\langle A,B\rangle=\Re\Tr(AB^{\ast})$. Starting with  $i_{(A,t\Phi)}$, its  adjoint  is the operator
\begin{eqnarray*}
\nonumber i_{(A,t\Phi)}^{\ast}\colon\Omega^1(\Sigma,\mathfrak{su}(E))\oplus\Omega^{1,0}(\Sigma,\mathfrak{sl}(E))\to \Omega^0(\Sigma,\mathfrak{su}(E)),\\
\label{eq:iast} i_{(A,t\Phi)}^{\ast}(\alpha,\varphi)=d_A^{\ast}\alpha+\frac{i}{2}\ast[t\Phi\wedge\varphi^{\ast}]-\frac{i}{2}\ast[t\Phi^{\ast}\wedge\varphi].
\end{eqnarray*}
The   adjoint of $L_{(A,\Phi)}$ is
\begin{eqnarray*}
L_{(A,t\Phi)}^{\ast}=\begin{pmatrix}d_A^{\ast}&tS_{\Phi}^{\ast}\\tR_{\Phi}^{\ast}&\bar\partial_A^{\ast}\end{pmatrix}.
\end{eqnarray*}
To compute the entries $R_{\Phi}^{\ast}$ and $S_{\Phi}^{\ast}$ we use the identities
\begin{equation*}
\ast (dz\wedge d\bar z)=-2i,\qquad\ast dz=-idz,\qquad \ast d\bar z=id\bar z.
\end{equation*}
Then we calculate the adjoints of the operators
\begin{align*}
M_{\Phi}\colon \Omega^{0,1}(\Sigma,\mathfrak{sl}(E))\to\Omega^{1,1}(\Sigma,\mathfrak{sl}(E)),\quad \psi\mapsto[\Phi\wedge\psi],\\
M_{\Phi^{\ast}}\colon \Omega^{1,0}(\Sigma,\mathfrak{sl}(E))\to\Omega^{1,1}(\Sigma,\mathfrak{sl}(E)),\quad \psi\mapsto[\Phi^{\ast}\wedge\psi].
\end{align*}
Writing in local coordinates  $\Phi=\varphi\, dz$, $\sigma=\underline\sigma\, dz\wedge d\bar z$ and $M_{\Phi}^{\ast}\sigma=\underline\tau\, d\bar z$ we find that for all $\psi=\underline\psi \,d\bar z$ the defining equation $\langle M_{\Phi}^{\ast}\sigma,\psi\rangle=\langle\sigma,M_{\Phi}\psi\rangle$ for the adjoint is equivalent to
\begin{equation*}
2\langle\underline\tau,\underline\psi\rangle=4\langle\underline\sigma,[\varphi,\underline\psi]\rangle,
\end{equation*}
hence
\begin{equation*}
\langle\underline\tau,\underline\psi\rangle=-2\langle[\underline\sigma,\varphi^{\ast}],\underline\psi\rangle,
\end{equation*}
so that
\begin{equation*}
\underline\tau=2[\varphi^{\ast},\underline\sigma].
\end{equation*}
Since $\underline\sigma=\frac{i}{2}\ast\sigma$ it follows that
\begin{equation*}
M_{\Phi}^{\ast}\sigma=i[\Phi^{\ast}\wedge \ast\sigma].
\end{equation*}
Similarly,
\begin{equation*}
M_{\Phi^{\ast}}^{\ast}\sigma=-i[\Phi \wedge \ast\sigma].
\end{equation*}
Let $i_{\C}\colon \Omega^1(\Sigma,\mathfrak{su}(E))\to \Omega^1(\Sigma,\mathfrak{sl}(E))$ and $i^{0,1}\colon \Omega^{0,1}(\Sigma,\mathfrak{sl}(E))\to \Omega^1(\Sigma,\mathfrak{sl}(E))$ be the natural inclusion maps, and let $\pi_{\C}$ and $\pi^{0,1}$ denote their adjoints. Then writing  $S_{\Phi}=M_{\Phi}\circ\pi^{0,1}\circ i_{\C}$ it follows that 
\begin{equation*}
S_{\Phi}^{\ast}=\pi_{\C}\circ i^{0,1}\circ M_{\Phi}^{\ast}\colon \Omega^{1,1}(\Sigma,\mathfrak{sl}(E))\to\Omega^1(\Sigma,\mathfrak{su}(E)),
\end{equation*}
hence
\begin{equation*}
S_{\Phi}^{\ast}\sigma=\frac{i}{2}[\Phi^{\ast}\wedge \ast\sigma] +\frac{i}{2}[\Phi^{\ast}\wedge \ast\sigma]^{\ast} =\frac{i}{2}[\Phi^{\ast}\wedge\ast\sigma]-\frac{i}{2}[\Phi\wedge\ast\sigma^{\ast}].
\end{equation*}
Furthermore, $R_{\Phi}=2\pi_{\C}\circ M_{\Phi^{\ast}}$, from which it follows that $R_{\Phi}^{\ast}=2M_{\Phi^{\ast}}^{\ast}\circ i_{\C}\colon  \Omega^{1,1}(\Sigma,\mathfrak{su}(E))\to\Omega^{1,0}(\Sigma,\mathfrak{sl}(E))$,  thus
\begin{equation*}
R_{\Phi}^{\ast}\mu=-2i[\Phi \wedge\ast\mu].
\end{equation*}
Inserting these computations, we arrive at
\begin{multline}\label{eq:operatorLast}
L_{(A,t\Phi)}^{\ast}\begin{pmatrix}\mu\\\sigma\end{pmatrix}=
\begin{pmatrix}d_A^{\ast}&tS_{\Phi}^{\ast}\\tR_{t\Phi}^{\ast}&\bar\partial_A^{\ast}
\end{pmatrix}\begin{pmatrix}\mu\\\sigma\end{pmatrix}\\
=\begin{pmatrix}d_A^{\ast}\mu+\frac{i}{2}t[\Phi^{\ast}\wedge\ast\sigma]-\frac{i}{2}t[\Phi\wedge\ast\sigma^{\ast}]\\ \bar\partial_A^{\ast}\sigma-2it[\Phi \wedge\ast\mu]  
\end{pmatrix}.
\end{multline}

\subsubsection*{{\bf The Laplacians}}
For the remainder of the article, if not mentioned otherwise, we assume that the pair $(A,t\Phi)$ is a solution to the self-duality equations $\mathcal H(A,t\Phi)=0$. The    Laplace operator $D_{(A,t\Phi)}^0$ in  \eqref{eq:laplace0} then takes the form
\begin{equation*}\label{eq:laplace1}
D_{(A,t\Phi)}^0\gamma=\Delta_A\gamma+t^2\ast[\Phi\wedge[\Phi^{\ast}\wedge\gamma]]-t^2\ast[\Phi^{\ast}\wedge[\Phi\wedge\gamma]].
\end{equation*}
This operator has been  studied to considerable extend in \cite{msww14} (there as an operator acting on sections of  $i\mathfrak{su}(E)$ rather than  $\mathfrak {su}(E)$, which is essentially the same since multiplication by $i$ intertwins both operators). Hence in the following we need to focus more closely on the second Laplacian $D_{(A,t\Phi)}^2$ as defined in \eqref{eq:laplace2}.  Using the expression in \eqref{eq:iast}, it reads
\begin{equation*}
D_{(A,t\Phi)}^2 =\begin{pmatrix}\Delta_A +t^2R_{\Phi}R_{\Phi}^{\ast}&td_AS_{\Phi}^{\ast}+tR_{\Phi}\bar\partial_A^{\ast}\\t S_{\Phi}d_A^{\ast}+t\bar\partial_A R_{\Phi}^{\ast}&\bar\partial_A\bar\partial_A^{\ast}+t^2S_{\Phi}S_{\Phi}^{\ast}\end{pmatrix}.
\end{equation*}
With
\begin{eqnarray*}
R_{\Phi}R_{\Phi}^{\ast}\mu&=&
2i[\Phi\wedge[\Phi^{\ast}\wedge\ast\mu]]-2i[\Phi^{\ast}\wedge[\Phi\wedge\ast\mu]],\\
\bar\partial_AR_{\Phi}^{\ast}\mu&=&
-2[\Phi\wedge\partial_A^{\ast}\mu]=-2S_{\Phi}d_A^{\ast}\mu,\\
S_{\Phi}S_{\Phi}^{\ast}\sigma&=&\frac{i}{2}[\Phi \wedge[\Phi^{\ast}\wedge\ast\sigma]],\\
R_{\Phi}\bar\partial_A^{\ast}\sigma&=&
[\Phi\wedge\partial_A^{\ast}\sigma^{\ast}]+[\Phi^{\ast}\wedge\bar\partial_A^{\ast}\sigma],\\
d_AS_{\Phi}^{\ast}\sigma&=&-\frac{i}{2}[\Phi^{\ast}\wedge\partial_A\ast\sigma]+\frac{i}{2}[\Phi\wedge\bar\partial_A\ast\sigma^{\ast}]\\
&=&-\frac{1}{2}R_{\Phi}\bar\partial_A^{\ast}\sigma
\end{eqnarray*}
we obtain that
\begin{multline}\label{eq:laplace2A}
D_{(A,t\Phi)}^2\begin{pmatrix}\mu\\\sigma\end{pmatrix}=\\
\begin{pmatrix}\Delta_A\mu +2it^2[\Phi\wedge[\Phi^{\ast}\wedge\ast\mu]]-2it^2[\Phi^{\ast}\wedge[\Phi\wedge\ast\mu]]&\frac{1}{2}t [\Phi^{\ast}\wedge \bar\partial_A^{\ast}\sigma]+\frac{1}{2}t[\Phi\wedge\partial_A^{\ast}\sigma^{\ast}]\\-t[\Phi\wedge\partial_A^{\ast}\mu]&\bar\partial_A\bar\partial_A^{\ast}\sigma+\frac{i}{2}t^2[\Phi \wedge[\Phi^{\ast}\wedge\ast\sigma]]\end{pmatrix}.
\end{multline}
In order to understand the asymptotic behaviour  of the individual terms in  the sectional curvature formula  \eqref{eq:sectionalcurvature}, we need to analyze  the family of  Green operators $G_{(A,t\Phi)}^2:= (D_{(A,t\Phi)}^2)^{-1}$ in the limit $t\to\infty$ (the analysis of $G_{(A,t\Phi)}^0:= (D_{(A,t\Phi)}^0)^{-1}$ being carried out in \cite{msww14}). This is the goal of the next section.

\subsection{Estimates on the Green operator $G_{(A,t\Phi)}^2$}

For a solution $(A,t\Phi)$ as above, we denote $i_t:=i_{(A,t\Phi)}$, $L_t:=L_{(A,t\Phi)}$   and $D_t^j:=D_{(A,t\Phi)}^j$, $j=0,2$. We first consider the operator $D_t^2$ on the disk $\D$, assuming   that $(A,t\Phi)=(A_t^{\fid},t\Phi_t^{\fid})$ is the fiducial solution as introduced in \textsection \ref{subsect:hitchinsection}.

\begin{proposition}\label{prop:kernelLast}
Suppose that  the pair $(\mu,\sigma)$  satisfies Neumann boundary conditions on $\D$. Then $D_t^2(\mu,\sigma)=0$ for any $t>0$ implies $(\mu,\sigma)=0$.
\end{proposition}

\begin{proof}
The equation $D_t^2(\mu,\sigma)=0$ and Neumann boundary conditions imply that $L_t^{\ast}(\mu,\sigma)=0$. Hence  $d_A^{\ast}\mu+\frac{i}{2}t[\Phi^{\ast}\wedge\ast\sigma]-\frac{i}{2}t[\Phi\wedge\ast\sigma^{\ast}]=0$. Differentiating this equation we obtain that
\begin{eqnarray*}
0&=& d_A d_A^{\ast}\mu-\frac{i}{2}t[\Phi^{\ast}\wedge \partial_A\ast\sigma]+\frac{i}{2}t[\Phi\wedge \bar\partial_A\ast\sigma^{\ast}]\\
&=&  d_A d_A^{\ast}\mu-\frac{i}{2}t[\Phi^{\ast}\wedge\ast  \bar{\partial}_A^{\ast}\sigma]+\frac{i}{2}t[\Phi\wedge \ast  \partial_A^{\ast}\sigma^{\ast}]\\
&=&  d_A d_A^{\ast}\mu +t^2[\Phi^{\ast}\wedge\ast   [\Phi\wedge \ast\mu]]+ t^2[\Phi\wedge \ast   [\Phi\wedge\ast\mu]^{\ast}],
\end{eqnarray*}
where the last inequality follows from $\bar{\partial}_A^{\ast}\sigma-2it[\Phi\wedge\ast\mu]=0$. Taking the inner product with $\mu$ and integrating by parts we get
\begin{eqnarray*}
0&=&\langle\mu,  d_A d_A^{\ast}\mu +t^2[\Phi^{\ast}\wedge\ast   [\Phi\wedge \ast\mu]]+ t^2[\Phi\wedge \ast   [\Phi\wedge\ast\mu]^{\ast}]\rangle\\
&=&\| d_A^{\ast}\mu \|_{L^2(\D)}^2+t^2\|[\Phi\wedge\ast\mu]\|_{L^2(\D)}^2+t^2\|[\Phi^{\ast}\wedge\ast\mu]\|_{L^2(\D)}^2. 
\end{eqnarray*}
The integration by parts is justified   again by the assumption that $\mu$ satisfies Neumann boundary conditions. Now  the latter equation implies that $[\Phi\wedge\ast\mu]=[\Phi^{\ast}\wedge\ast\mu]=0$, which we claim forces $\mu$ to vanish identically. Namely, writing $\Phi=\varphi\, dz$, this condition implies that $\ast\mu=u\varphi$ and $\ast\mu=v\varphi^{\ast}$ for   complex-valued functions $u$ and $v$. Since for all $(r,\theta)\in\D$, $r\neq0$,
\begin{equation*}
\varphi(r,\theta)=r^{\frac{1}{2}}\begin{pmatrix}
0& e^{-h_t}e^{i\theta}\\ e^{h_t}&0\end{pmatrix}\qquad\textrm{and}\qquad \varphi^{\ast}(r,\theta)=r^{\frac{1}{2}}\begin{pmatrix}
0& e^{h_t}\\ e^{-h_t}e^{-i\theta}&0 
\end{pmatrix}
\end{equation*}
are linearly independent, we conclude  that $u=v=0$ and hence that $\mu$   vanishes identically. It is then not difficult to see that also $\sigma=0$. Namely,  with $\mu=0$ it follows that $\bar{\partial}_A^{\ast}\sigma=0$ and $[\Phi^{\ast}\wedge\ast\sigma]=0$. The latter condition shows that $\ast\sigma=u\varphi^{\ast}$ for some further complex-valued function $u$. The first condition is equivalent to $\partial_A\ast\sigma=0$. We therefore conclude that $0=\partial_A(u\varphi^{\ast})= \partial u\cdot \varphi^{\ast}$, since $\partial_A\Phi^{\ast}=0$. Hence the function $u$ is antiholomorphic on $\D$. Together with Neumann boundary conditions assumed by $\ast\sigma$ it is straightforward to check that $u$  has to vanish identically. This completes the proof.
\end{proof}

We use this proposition to  establish a uniform lower bound for the first eigenvalue of $D_t^2$. For the operator   $D_t^0$ such a lower bound has been shown in  \cite[Lemma 6.3]{msww14}. We follow its proof, which is based on the domain decomposition principle as stated in \cite{ba00}, with some minor modifications. At this point we also need to introduce, for any limiting configuration $(A_{\infty},\Phi_{\infty})$, the  splitting of the vector bundle $\mathfrak{su}(E)\to\Sigma^{\times}$ into the  direct sum of the line bundle $L_{\infty}$ of traceless skew-hermitian endomorphisms commuting with $\Phi_{\infty}$ and its $L^2$ orthogonal complement $L_{\infty}^{\perp}$. Both subbundles are parallel with respect to the connection $A_{\infty}$, cf.\ \cite[\textsection 4.2]{msww14} for details. Note that the complexifications of these subbundles give rise to the decomposition $\mathfrak{sl}(E)=L_{\infty}^{\C}\oplus L_{\infty}^{\C,\perp}$.

\begin{lemma}\label{lem:uniformlowerboundeigenvalue}
There exists a constant $C>0$ such that the  smallest eigenvalue $\lambda_{t}$ of $D_t^2$ satisfies $\lambda_t\geq C$ for all sufficiently large $t$. 
\end{lemma}

\begin{proof}
For $t\geq1$ we decompose $\Sigma$ into the disjoint union  $\Sigma=\bar{\Sigma}_t^{\operatorname{ext}}\cup \bigcup_{p\in\mathfrak p}\D_{p,t}$, where $\D_{p,t}$ denotes the open disk of radius $ct^{-\frac{2}{3}}$ about $p$ and $\Sigma_t^{\operatorname{ext}}=\Sigma\setminus \bigcup_{p\in\mathfrak p}\bar{\D}_{p,t}$. The domain decomposition principle  yields   for $\lambda_{t}$ the lower bound
\begin{equation*}
\lambda_{t}\geq\min_{p\in\mathfrak p} \{\lambda_t(\Sigma_t^{\operatorname{ext}}) ,\lambda_t(\D_{p,t})\},
\end{equation*}
where $\lambda_t(U)$ denotes the smallest eigenvalue of $D_t^2$ under Neumann boundary conditions on the subdomain $U\in\{ \Sigma_t^{\operatorname{ext}}\}\cup  \bigcup_{p\in\mathfrak p}\{\D_{p,t}\}$. We show a uniform lower bound for  these, utilizing in each case the variational characterization of the smallest Neumann eigenvalue  as the infimum of the Rayleigh quotient
\begin{equation*}
\mathcal R_{U,t}(v):=\frac{\|L_t^{\ast}v\|_{L^2(U)}^2}{\|v\|_{L^2(U)}^2}
\end{equation*}
over all nonzero $v\in H^1(U)$. The result is then implied by the following two claims.

\setcounter{step}{0}
\begin{claim}
There is a constant $C>0$ such that for all $p\in\mathfrak p$ the smallest Neumann eigenvalue on $\D_{p,t}$ satisfies  $\lambda_t(\D_{p,t})\geq C$ for all sufficiently large $t$.
\end{claim}

Since for $t$ sufficiently large, the solution $(A,t\Phi)$ differs from $(A_t^{\fid},t\Phi_t^{\fid})$ by some exponentially small (w.r.t.\ any $C^k$ norm) term, it suffices to prove the claim with $(A,t\Phi)$ replaced by $(A_t^{\fid},t\Phi_t^{\fid})$. We further observe that the numerator of the Rayleigh quotient $\mathcal R_{\D_{p,t},t}$ is invariant under the conformal rescaling $  (r,\theta)\mapsto (\rho,\theta)$ with $\rho=\frac{8}{3} tr^{\frac{3}{2}}$. In fact, a straightforward calculation shows that for every $v\in H^1(\D_{p,t})$ the quantity
$| L_t^{\ast} v|^2$ scales with the factor $t^{\frac{4}{3}}$, while  the area form $rdrd\theta=\frac{\sqrt[3]{3}}{8} t^{-\frac{4}{3}}\rho^{\frac{1}{3}}d\rho d\theta$ scales with $t^{-\frac{4}{3}}$. As for the denominator $\|v\|_{L^2(\D_{p,t})}^2$, we get a scaling with the factor $t^{-\frac{4}{3}}$, so that altogether
\begin{equation*}
\mathcal R_{\D_{p,t},t}(v)= t^{\frac{4}{3}} \mathcal R_{\D_{p,1},1}(\hat v),
\end{equation*}
where $\hat v(\rho,\theta)=v(r,\theta)$. Passing to the infimum on both sides it follows that $\lambda_t(\D_{p,t})= t^{\frac{4}{3}}\lambda_1(\D_{p,1})$. Since by Proposition \ref{prop:kernelLast} the kernel of the operator $L_1^{\ast}$ under Neumann boundary conditions is trivial, it follows that $\lambda_1(\D_{p,1})>0$, completing the proof of the claim.

\begin{claim}
There is a constant $C>0$ such that   the smallest Neumann eigenvalue on $ \Sigma_t^{\operatorname{ext}}$ satisfies  $\lambda_t( \Sigma_t^{\operatorname{ext}})\geq C$  for all    sufficiently large $t$.
\end{claim}

Since on $\Sigma_t^{\operatorname{ext}}$ the pair $(A_t,\Phi_t)$ differs from $(A_{\infty},\Phi_{\infty})$ by a  term which decays (w.r.t.~any $C^k$ norm)  exponentially to $0$ as $t\to\infty$, it is enough to show a $t$-independent lower bound for  $\mathcal R_{\D_{p,t},t}(v)$ where in the Rayleigh quotient the  operator $L_t^{\ast}$ is replaced by the one induced by  $(A_{\infty},\Phi_{\infty})$. We again denote this new operator by $L_t^{\ast}$.   From \eqref{eq:operatorLast} we see that $L_t^{\ast}$ acts invariantly on the subspaces   $\Omega^2(\Sigma,L_{\infty})\oplus \Omega^2(\Sigma,L_{\infty}^{\C})$ and $\Omega^2(\Sigma,L_{\infty}^{\perp})\oplus \Omega^2(\Sigma,L_{\infty}^{\C,\perp})$, so that  it is enough to show a $t$-independent lower bound for  the restriction of the Rayleigh quotient to either of these subspaces. Note that $L_t^{\ast}\colon (\mu,\sigma)\mapsto (d_{A_{\infty}}^{\ast}\mu, \bar{\partial}_{A_{\infty}}^{\ast}\sigma)$ for any element  $(\mu,\sigma)$ of  the former subspace. Since the connection $A_{\infty}$ is flat it follows that ${\partial}_{A_{\infty}}^{\ast}\bar{\partial}_{A_{\infty}}^{\ast}=\frac{1}{2}\Delta_{A_{\infty}}$, so that it suffices to argue that $\Delta_{A_{\infty}}$ satisfies a $t$-independent positive lower bound on $\Sigma_t^{\operatorname{ext}}$ under Neumann boundary conditions. This   follows by the same line of argument as in the proof of \cite[Proposition 5.2 (i)]{msww14}. It makes use  of the fact that on the line bundle $L_{\infty}$ the Laplacian $\Delta_{A_{\infty}}$ equals a nonnegative operator plus a potential term satisfying a uniform positive pointwise lower bound, and carries over to the situation at hand. Next we consider $L_t^{\ast}$ as an operator on   $\Omega^2(\Sigma,L_{\infty}^{\perp})\oplus \Omega^2(\Sigma,L_{\infty}^{\C,\perp})$.  Using the expression \eqref{eq:laplace2A} for $D_t^2=L_tL_t^{\ast}$ together with the Cauchy-Schwarz inequality we can in this case estimate the numerator of the Rayleigh quotient $\mathcal R_{\Sigma_t^{\operatorname{ext}}}(\mu,\sigma)$ from above as (using the notation $(A,\Phi)=(A_{\infty},\Phi_{\infty})$ for short)
\begin{eqnarray*}
\lefteqn{\langle (\mu,\sigma), D_t^2(\mu,\sigma)\rangle_{L^2(\Sigma_t^{\operatorname{ext}})}}\\
&=& \|d_A\mu\|_{L^2(\Sigma_t^{\operatorname{ext}})}^2 + \|\bar\partial_A^{\ast}\sigma\|_{L^2(\Sigma_t^{\operatorname{ext}})}^2 \\
&&+2t^2\langle\mu,i[\Phi\wedge[\Phi^{\ast}\wedge\ast\mu]]- i[\Phi^{\ast}\wedge[\Phi\wedge\ast\mu]]\rangle_{L^2(\Sigma_t^{\operatorname{ext}})}\\
&&+\frac{t}{2}\langle\mu, [\Phi^{\ast}\wedge\bar\partial_A^{\ast}\sigma]
+[\Phi\wedge\partial_A^{\ast}\sigma^{\ast}]\rangle_{L^2(\Sigma_t^{\operatorname{ext}})}\\
&&+\frac{t^2}{2}\langle\sigma,i[\Phi\wedge[\Phi^{\ast}\wedge\ast\sigma]]\rangle -t\langle\sigma,[\Phi\wedge\partial_A^{\ast}\mu]\rangle_{L^2(\Sigma_t^{\operatorname{ext}})}\\
&=&\|d_A\mu\|_{L^2(\Sigma_t^{\operatorname{ext}})}^2 + \|\bar\partial_A^{\ast}\sigma\|_{L^2(\Sigma_t^{\operatorname{ext}})}^2\\
&& +4t^2\| [\Phi\wedge\ast\mu]\|_{L^2(\Sigma_t^{\operatorname{ext}})}^2+\frac{t^2}{2}\|[\Phi^{\ast}\wedge\ast\sigma]\|_{L^2(\Sigma_t^{\operatorname{ext}})}^2\\
&&- \frac{t}{2}\langle [\ast\mu\wedge\Phi],\bar\partial_A^{\ast}\sigma\rangle_{L^2(\Sigma_t^{\operatorname{ext}})} - \frac{t}{2}\langle[\ast\mu\wedge\Phi^{\ast}],\partial_A^{\ast}\sigma^{\ast}\rangle_{L^2(\Sigma_t^{\operatorname{ext}})}\\
&&+t\langle[\ast\sigma\wedge\Phi^{\ast}],\partial_A^{\ast}\mu\rangle_{L^2(\Sigma_t^{\operatorname{ext}})}\\
&\geq& \|d_A\mu\|_{L^2(\Sigma_t^{\operatorname{ext}})}^2 + \|\bar\partial_A^{\ast}\sigma\|_{L^2(\Sigma_t^{\operatorname{ext}})}^2\\
&& +4t^2\| [\Phi\wedge\ast\mu]\|_{L^2(\Sigma_t^{\operatorname{ext}})}^2+\frac{t^2}{2}\|[\Phi^{\ast}\wedge\ast\sigma]\|_{L^2(\Sigma_t^{\operatorname{ext}})}^2\\
&&- \frac{1}{4} \|\bar\partial_A^{\ast}\sigma\|_{L^2(\Sigma_t^{\operatorname{ext}})}^2-t^2\| [\Phi\wedge\ast\mu]\|_{L^2(\Sigma_t^{\operatorname{ext}})}^2\\
&&- \|\partial_A^{\ast}\mu\|_{L^2(\Sigma_t^{\operatorname{ext}})}^2-\frac{t^2}{4}\|[\Phi^{\ast}\wedge\ast\sigma]\|_{L^2(\Sigma_t^{\operatorname{ext}})}^2\\
&=& \|\bar \partial_A^{\ast}\mu\|_{L^2(\Sigma_t^{\operatorname{ext}})}^2 +\frac{3}{4} \|\bar\partial_A^{\ast}\sigma\|_{L^2(\Sigma_t^{\operatorname{ext}})}^2\\
&& +3t^2\| [\Phi\wedge\ast\mu]\|_{L^2(\Sigma_t^{\operatorname{ext}})}^2+\frac{t^2}{4}\|[\Phi^{\ast}\wedge\ast\sigma]\|_{L^2(\Sigma_t^{\operatorname{ext}})}^2.
\end{eqnarray*}
Since   sections of  $L_{\infty}^{\perp}$, respectively of $L_{\infty}^{\C,\perp}$ satisfy  the uniform pointwise lower bounds
\begin{equation*}
| [\Phi\wedge\ast\mu]|^2\geq C|\mu|^2\qquad\textrm{and}\qquad     |[\Phi^{\ast}\wedge\ast\sigma]|^2\geq C|\sigma|^2
\end{equation*}
for some $t$-independent constant $C>0$ (cf.\ \cite{msww14} for details), we get that
\begin{equation*}
\langle (\mu,\sigma), D_t^2(\mu,\sigma)\rangle_{L^2(\Sigma_t^{\operatorname{ext}})} \geq Ct^2 \|(\mu,\sigma)\|_{L^2(\Sigma_t^{\operatorname{ext}})}^2,
\end{equation*}
and the claim follows.
\end{proof}

As an immediate consequence, we record the following corollary.

\begin{corollary}\label{cor:uniformupperbound}
There exists a constant $C>0$ such that the norms of the operators $G_{t}^0$   and $G_{t}^2$ satisfy the uniform bound
\begin{equation*}
\|G_{t}^0\|_{\mathcal L(L^2(\Sigma))}+\|G_{t}^2\|_{\mathcal L(L^2(\Sigma))}\leq C
\end{equation*}
for all sufficiently large $t$. 
\end{corollary}

\begin{remark}
The estimate in Corollary \ref{cor:uniformupperbound} holds  for all $t\geq1$, as follows from a simple compactness argument and the fact (shown in \cite{hi87})  that the operators $D_t^j$, $j=0,2$,  have bounded inverses $G_{t}^j$ for all $t>0$.
\end{remark}

Subsequently, we let  $\Delta$ denote the (negative semidefinite) Laplace-Beltrami operator on $\Sigma$. Thus  $\Delta  =-d^{\ast}d= -2\bar\partial^{\ast}\bar\partial$.

\begin{lemma}\label{lem:subharmonicest}
Let $U\subseteq\Sigma$ be open and suppose that $(\mu,\sigma)$ satisfies $D_t^2(\mu,\sigma)=0$ on $U$. Then the function $u=\frac{1}{2}(|\mu |^2+\frac{3}{4}| \sigma|^2)$ satisfies the differential inequality
\begin{multline}\label{eq:diffineqsubharm}
\Delta u\geq  \frac{1}{7} | \partial_{A}^{\ast}\mu   |^2+ | \bar\partial _A^{\ast}\mu |^2 +\frac{3}{2}  \left| \bar\partial_A^{\ast}\sigma\right|^2\\
+\frac{7}{2} t^2|[ \Phi\wedge \ast\mu]|^2 +\frac{3}{32} t^2|[\Phi^{\ast}\wedge \ast\sigma]|^2.
\end{multline}
It is in particular subharmonic on $U$ and hence assumes its maximum on the boundary $\partial U$. An analogue statement holds for the operator $D_t^0$ (cf.\ \cite{msww17} for the slightly easier proof).
\end{lemma}

\begin{proof}
The function $\Delta u$ satisfies   the general identity
\begin{equation*}
\Delta u=\frac{1}{2}\Delta (|\mu |^2+\frac{3}{4}| \sigma|^2)=\left| d_A^{\ast}\mu\right|^2-\langle d_Ad_A^{\ast}\mu,\mu\rangle+\frac{3}{2}\left| \bar\partial_A^{\ast}\sigma\right|^2-\frac{3}{2}\langle \bar\partial_A\bar\partial_A^{\ast}\sigma,\sigma\rangle
\end{equation*}
with respect to any unitary connection $A$. Now replacing   $d_Ad_A^{\ast}\mu$ and $\bar\partial_A\bar\partial_A^{\ast}\sigma$ using the equation    $D_t^2(\mu,\sigma)=0$ we obtain that
\begin{eqnarray*}\label{eq:Laplupos}\begin{split}
 \lefteqn{\Delta u=\left| d_A^{\ast}\mu\right|^2+\frac{3}{2}\left| \bar\partial_A^{\ast}\sigma\right|^2}\\
&&+\langle\mu,2it^2[\Phi\wedge[\Phi^{\ast}\wedge\ast\mu]]-2it^2[\Phi^{\ast}\wedge[\Phi\wedge\ast\mu]]\\
&&+\frac{t}{2}[\Phi^{\ast}\wedge\bar\partial_{A}^{\ast}\sigma]+\frac{t}{2}[\Phi\wedge\partial_{A  }^{\ast}\sigma^{\ast}]\rangle\\
&&+\frac{3}{2}\langle\sigma,-t[\Phi\wedge\partial_{A}^{\ast}\mu]+\frac{i}{2}t^2[\Phi\wedge[\Phi^{\ast}\wedge\ast\sigma]]\rangle.
\end{split}
\end{eqnarray*}
Using the identity $\langle \alpha,[\beta,\gamma]\rangle=-\langle[\alpha,\beta^{\ast}],\gamma\rangle$ satisfied for  the chosen hermitian inner product on $\mathfrak{sl}(2,\C)$ we may rewrite 
\begin{eqnarray*}
\lefteqn{  \langle\mu,2it^2[\Phi\wedge[\Phi^{\ast}\wedge\ast\mu]]-2it^2[\Phi^{\ast}\wedge[\Phi\wedge\ast\mu]]}\\
&&+\frac{t}{2}[\Phi^{\ast}\wedge\bar\partial_{A}^{\ast}\sigma]+\frac{t}{2}[\Phi\wedge\partial_{A}^{\ast}\sigma^{\ast}]\rangle \\
&=&2t^2|[\ast\mu\wedge \Phi^{\ast}]|^2+2t^2|[\ast\mu\wedge \Phi]|^2\\
&&+\frac{it}{2}\langle [\ast\mu\wedge\Phi],\bar\partial_{A}^{\ast}\sigma\rangle-\frac{it}{2}\langle[\ast\mu\wedge\Phi^{\ast}],\partial_{A}^{\ast}\sigma^{\ast}\rangle\\
&=&4t^2|[\ast\mu\wedge \Phi ]|^2 \\
&&+\frac{it}{2}\langle [\ast\mu\wedge\Phi],\bar\partial_{A}^{\ast}\sigma\rangle-\frac{it}{2}\langle[\ast\mu\wedge\Phi^{\ast}],\partial_{A}^{\ast}\sigma^{\ast}\rangle
\end{eqnarray*}
and
\begin{equation*}
\frac{3}{2}\langle\sigma,-t[\Phi\wedge\partial_{A}^{\ast}\mu]+\frac{i}{2}t^2[\Phi\wedge[\Phi^{\ast}\wedge\ast\sigma]]\rangle=\frac{3}{4}t^2|[\ast\sigma\wedge\Phi^{\ast}]|^2+\frac{3}{2}it\langle[\ast\sigma\wedge\Phi^{\ast}],\partial_{A}^{\ast}\mu\rangle.
\end{equation*}
An application of  the Cauchy-Schwarz inequality now
yields
\begin{multline*}
\frac{t}{2}\left|\langle [\ast\mu\wedge\Phi],\bar\partial_{A}^{\ast}\sigma\rangle\right|+\frac{t}{2}\left|\langle[\ast\mu\wedge\Phi^{\ast}],\partial_{A}^{\ast}\sigma^{\ast}\rangle\right|\\
\leq t|  [\ast\mu\wedge\Phi] | \cdot|  \bar\partial_{A}^{\ast}\sigma|\leq \frac{t^2}{2}  |  [\ast\mu\wedge\Phi]|^2+  \frac{1}{2} |  \bar\partial_{A}^{\ast}\sigma |^2
\end{multline*}
and (for any $\varepsilon>0$)
\begin{equation*}
\frac{3}{2}t\left|\langle[\ast\sigma\wedge\Phi^{\ast}],\partial_{A}^{\ast}\mu\rangle\right|\leq \frac{3}{4}\varepsilon^2t^2|[\ast\sigma\wedge\Phi^{\ast}]  |^2 +    \frac{3}{4}\varepsilon^{-2} | \partial_{A}^{\ast}\mu   |^2.
\end{equation*}
Collecting all terms and noting that  $| d_A^{\ast}\mu |^2=| \partial_A^{\ast}\mu |^2+| \bar\partial _A^{\ast}\mu |^2$ we arrive at the inequality
\begin{eqnarray*}
\Delta u&\geq& \left| d_A^{\ast}\mu\right|^2+\frac{3}{2}\left| \bar\partial_A^{\ast}\sigma\right|^2\\
&&+4t^2|[\ast\mu\wedge \Phi]|^2 - \frac{t^2}{2}  |  [\ast\mu\wedge\Phi]|^2  -  \frac{1}{2} |  \bar\partial_{A}^{\ast}\sigma |^2\\
&&+\frac{3}{4}t^2|[\ast\sigma\wedge\Phi^{\ast}]|^2  -    \frac{3}{4}\varepsilon^2t^2|[\ast\sigma\wedge\Phi^{\ast}]  |^2 -  \frac{3}{4}\varepsilon^{-2} | \partial_{A}^{\ast}\mu   |^2\\
&=&  (1-  \frac{3}{4}\varepsilon^{-2}) | \partial_{A}^{\ast}\mu   |^2+ | \bar\partial _A^{\ast}\mu |^2 +   \left| \bar\partial_A^{\ast}\sigma\right|^2\\
&& +\frac{7}{2} t^2|[\ast\mu\wedge \Phi]|^2 +\frac{3}{4}(1-\varepsilon^2)t^2|[\ast\sigma\wedge\Phi^{\ast}]|^2.  
\end{eqnarray*}
Upon choosing $\varepsilon^2=\frac{7}{8}$, the claimed inequality follows. 
\end{proof}

\begin{lemma}\label{lem:unifC0bounddisks}
For $t\geq1$, let $\xi_t$ be the solution of the equation $D_t^2\xi_t=\eta_t$, where $\supp \eta_t\Subset \Sigma^{\operatorname{ext}}$. We further assume the uniform bound    $\|\eta_t\|_{L^2(\Sigma)}\leq C_0$ for some constant $C_0$ and all $t$. Then there exists a constant $C$ such that $|\xi_t(x)| \leq C$ for all $x\in\bigcup_{p\in\mathfrak p}\D(p)$ and $t$. An equivalent statement holds with the operator $D_t^2$ replaced by $D_t^0$.
\end{lemma}

\begin{proof}
With   $\|\eta_t\|_{L^2(\Sigma)}\leq C_0$, Corollary \ref{cor:uniformupperbound} yields the uniform estimate $\|\xi_t\|_{L^2(\Sigma)}\leq C$ for some constant $C$ and all $t\geq1$. 
We fix a number $R>1$ such that $\D_R(p)\cap \supp \eta_t=\emptyset$ for all $p\in\mathfrak p$.
By Fubini's theorem, there is a constant $C_1$ and  for every $t$  a  constant $1\leq r(t)\leq R$   such that (with $S_{r(t)}(p)$ denoting the boundary of $\D_{r(t)}(p)$)  
\begin{equation*}
\int_{S_{r(t)}(p)}|\xi_t|^2\, d\theta \leq C_1 \|\xi_t\|_{L^2(\Sigma)}^2\leq C_2
\end{equation*}
for some further $t$-independent constant $C_2$. Since we assumed that $\xi_t$ satisfies $D_t^2\xi_t=0$ on $\D_R(p)$ for all $p\in\mathfrak p$, Lemma \ref{lem:subharmonicest} yields subharmonicity of  the function $u_t=\frac{1}{2}(|\mu_t |^2+\frac{3}{4}| \sigma_t|^2)$  on each disk $\D_{r(t)}(p)$, where $\xi_t=(\mu_t,\sigma_t)$. The   mean value property of subharmonic functions now implies the result.
\end{proof}

\subsection{Local analysis of the model equation}\label{subsect:localanalysis} 

We next consider  the family of model equations $r^2D_t^j\xi_t=\eta_t$ over the complex plane $\C$. Here we suppose that the operator $D_t^j$ is induced by the fiducial solution $(A_t^{\fid},\Phi_t^{\fid})$ and that the right-hand side  is of the form $\eta_t(re^{i\theta})=\eta(\rho e^{i\theta})$ for some $t$-idependent function $\eta$, where $\rho=\frac{8}{3} tr^{\frac{3}{2}}$. We analyze this equation by means of a suitable decomposition of the Hilbert space $H=L^2(\C)$ into an orthogonal  sum of $D_t^j$-invariant subspaces. This decomposition  permits us to reduce the model equation to a system of ordinary differential equations. Details are carried out for the operator   $D_t^2$; the easier operator $D_t^0$ has been analyzed by similar methods in \cite{msww14}.\\ 
\medskip\\
Let us  define the Hilbert subspaces $H^+_{\ell}$ and $H^-_{\ell}$ of $H$ consisting of the pairs of square integrable two-forms 
\begin{multline}\label{eq:pairhellplus}
\begin{pmatrix}
\mu\\\sigma\end{pmatrix}=  \left(  \begin{pmatrix}\bar\mu_{\ell}e^{-i\ell\theta}+ \mu_{\ell}e^{i\ell\theta} &0\\  0&  -\bar\mu_{\ell}e^{-i\ell\theta}- \mu_{\ell}e^{i\ell\theta} \end{pmatrix}dz\wedge d\bar z,\right.\\
\left. \begin{pmatrix}0&  \sigma_{-\ell+2}e^{i(-\ell+2)\theta}+ \sigma_{\ell+2}e^{i(\ell+2)\theta} \\    \tau_{-\ell+1}e^{i(-\ell+1)\theta}+ \tau_{\ell+1}e^{i(\ell+1)\theta}&0 \end{pmatrix}dz\wedge d\bar z\right),
\end{multline}
respectively
\begin{multline}\label{eq:pairhellminus}
\begin{pmatrix}
\mu\\\sigma\end{pmatrix}=  \left(  \begin{pmatrix}0&\mu_{\ell}e^{i\ell\theta}+ \mu_{-\ell+1}e^{i(-\ell+1)\theta} \\   \bar\mu_{\ell}e^{-i\ell\theta}+ \bar\mu_{-\ell+1}e^{-i(-\ell+1)\theta}&0\end{pmatrix}dz\wedge d\bar z,\right. \\
\left.\begin{pmatrix}  \sigma_{\ell+1}e^{i(\ell+1)\theta}+ \sigma_{-\ell+2}e^{i(-\ell+2)\theta} &0 \\   0&  -\sigma_{\ell+1}e^{i(\ell+1)\theta}- \sigma_{-\ell+2}e^{i(-\ell+2)\theta} \end{pmatrix}dz\wedge d\bar z\right),
\end{multline}
where $\mu_j, \sigma_{j},\tau_j\colon[0,\infty)\to\C$  are   functions of the radial variable $r$.  It is straightforward to check that $D_t^2$ preserves the $L^2$ orthogonal decomposition
\begin{equation*}
H=\bigoplus_{\ell\geq0}H_{\ell}^+\oplus\bigoplus_{\ell\geq1}H_{\ell}^-.
\end{equation*}
For ease of notation, we identify the pair of two-forms in  \eqref{eq:pairhellplus} with the tuple of functions $(\mu_\ell,\sigma_{\ell+2},\sigma_{-\ell+2},\tau_{\ell+1},\tau_{-\ell+1})$, and similarly for the pair in \eqref{eq:pairhellminus}. We denote by $D_{t,\ell}^{\pm}$ the restriction of $D_t^2$ to  $H^{\pm}_{\ell}$.   It satisfies
\begin{multline}\label{eq:opDtj+}
D_{t,\ell}^+\begin{pmatrix}
\mu_\ell\\ \sigma_{\ell+2}\\\sigma_{-\ell+2} \\ \tau_{\ell+1}\\ \tau_{-\ell+1}
\end{pmatrix}=-\frac{1}{2r^2}(r\partial_r)^2\begin{pmatrix}
2\mu_\ell\\ \sigma_{\ell+2}\\\sigma_{-\ell+2} \\ \tau_{\ell+1}\\ \tau_{-\ell+1}
\end{pmatrix} \\
+\frac{1}{2r^2}\begin{pmatrix}
2  \ell^2\mu_\ell+64t^2r^3\cosh(2h_t)\mu_\ell\\
\big(\ell+2-4f_t\big)^2\sigma_{\ell+2} +4rf_t'\sigma_{\ell+2}+4t^2r^3e^{-2h_t}\sigma_{\ell+2}-4t^2r^3\tau_{\ell+1}\\
\big(-\ell+2-4f_t\big)^2\sigma_{-\ell+2} +4rf_t'\sigma_{-\ell+2}+4t^2r^3e^{-2h_t}\sigma_{-\ell+2}-4t^2r^3\tau_{-\ell+1}\\
\big(\ell+1+4f_t\big)^2\tau_{\ell+1} -4rf_t'\tau_{\ell+1}-4t^2r^3\sigma_{\ell+2}+4t^2r^3e^{2h_t}\tau_{\ell+1}\\
\big(-\ell+1+4f_t\big)^2\tau_{-\ell+1} -4rf_t'\tau_{-\ell+1}-4t^2r^3\sigma_{-\ell+2}+4t^2r^3e^{2h_t}\tau_{-\ell+1}
\end{pmatrix}\\
+\frac{tr^{\frac{1}{2}}}{2}\begin{pmatrix}
  e^{-h_t}\big\{\sigma_{\ell+2}'+\frac{\ell+2}{r}\sigma_{\ell+2}-\frac{4f_t}{r}\sigma_{\ell+2}+\sigma_{-\ell+2}'+\frac{-\ell+2}{r}\sigma_{-\ell+2} \\
 -\frac{4f_t}{r}\sigma_{-\ell+2}\big\} -   e^{h_t}\big\{\tau_{\ell+1}'+\frac{\ell+1}{r}\tau_{\ell+1}+\frac{4f_t}{r}\tau_{\ell+1} +\tau_{-\ell+1}'+\frac{-\ell+1}{r}\tau_{-\ell+1}+\frac{4f_t}{r}\tau_{-\ell+1}\big\}  \\ -4 e^{-h_t}(\mu_\ell'-\frac{\ell}{r}\mu_\ell) \\
-4 e^{-h_t}(\bar\mu_\ell'+\frac{\ell}{r}\bar\mu_\ell)\\
4 e^{h_t}( \mu_\ell'-\frac{\ell}{r}\mu_\ell)\\
4 e^{-h_t}(\bar\mu_\ell'+\frac{\ell}{r}\bar\mu_\ell)
\end{pmatrix}.
\end{multline}
The operator $D_{t,\ell}^+$ in this explicit form will be used below. A similar expression holds for $D_{t,\ell}^-$, which we do not need to write out here. We now turn  to the model equation $r^2D_t^2\xi_t=\eta_t$, which we analyze by making crucial use of the invariance of the operator $r^2D_t^2$ under the conformal rescaling $(r,\theta)\mapsto (\rho,\theta)=(\frac{8}{3}tr^{\frac{3}{2}},\theta)$.  Recall that we assume    the right-hand side to be of the form $\eta_t(re^{i\theta})=\eta(\rho e^{i\theta})$ for some $t$-idependent function $\eta$. Then under this rescaling, the model problem turns  into the equation
\begin{equation*}\label{eq:resclaedmodeq}
\rho^2D_{\ast}^2\xi=\eta
\end{equation*}
for some $t$-independent differential operator $D_{\ast}^2$ (which is not necessary to write out here explicitly). On each subspace $H_\ell^{\pm}$ this equation reduces to the $t$-independent system of second order ODEs
\begin{equation}\label{eq:order2system}
\rho^2D_{\ast,\ell}^{\pm}\xi_\ell^{\pm}=\eta_\ell^{\pm},
\end{equation}
where $\xi_\ell^{\pm}$ and $\eta_\ell^{\pm}$ denote the component of $\xi$, respectively of $\eta$, in $H_\ell^{\pm}$.  Introducing  $\rho\partial_{\rho}\xi_\ell^{\pm}$  as a new unknown function turns \eqref{eq:order2system} into a system of first order ODEs,    which we may write as  $\rho\hat D_{\ast,\ell}^{\pm}\xi_\ell^{\pm}=\eta_\ell^{\pm}$, or equivalently as
\begin{equation}\label{eq:FouriercompODE}
\hat D_{\ast,\ell}^{\pm}\xi_\ell^{\pm}=\frac{\eta_\ell^{\pm}}{\rho}.
\end{equation}
Let $\Phi_\ell^{\pm}$
be a fundamental system of solutions of the homogeneous equation $\hat D_{\ast,\ell}^{\pm}\xi_\ell^{\pm}=0$ on $\C$. By variation of constants,  a particular solution of Eq.\ \eqref{eq:FouriercompODE} is
\begin{equation}\label{eq:varofconstsolution}
\xi_\ell^{\pm}(\rho)= \Phi_\ell^{\pm}(\rho)\int_0^{\rho} (\Phi_\ell^{\pm})^{-1}(\lambda) \eta_\ell^{\pm}(\lambda)\, \frac{d\lambda}{\lambda}.
\end{equation}
Returning to our original PDE $r^2D_t^2\xi_t=\eta_t$, we have thus obtained a solution $\xi_t$ of the form $\xi_t(re^{i\theta})=\xi_{\ast}(\rho e^ {i\theta})$, where the component  of  $\xi_t$ in $H_\ell^{\pm}$  is   $\xi_{t,\ell}^{\pm}(r)=\xi_\ell^{\pm}(\rho)$. 
\medskip\\
Subsequently, we are concerned with the model equation  $r^2D_t^\ell\xi_t= \eta_t$ where the support of   the  function $\eta_t$ is contained in the disk of radius $Ct^{- \frac{2}{3}}$ about $0$, respectively  $\eta$ is supported in the  disk of $t$-independent radius $C$ about $0$. In this situation,  we may add an appropriate solution of the homogeneous equation and therefore  arrange for the solution in \eqref{eq:varofconstsolution}  to be in the domain  $\{\rho e^{i\theta}\in\C\mid \rho\geq C\}$ of the form
\begin{equation*}
\xi_\ell^{\pm}(\rho)=  B(\eta_\ell^{\pm})  \varphi_{\ell}^{\pm}(\rho),
\end{equation*}
where   $\varphi_{\ell}^{\pm}$ is any   solution of the homogeneous equation $D_{\ast,\ell}^{\pm}\xi_\ell^{\pm}=0$, and the map $\eta_\ell^{\pm}\mapsto  B(\eta_\ell^{\pm})\in\C$ is linear. We are thus lead to consider more closely this homogeneous equation in the region  $\rho\geq C$ for  sufficiently large constant $C$. We  show for each $\ell$ the existence of a  solution $\varphi_{\ell}^{\pm}$ with      at least polynomial decay rate in $\rho$. 
For  $\rho\geq C$ the pair $(A_t^{\fid},\Phi_t^{\fid})$ equals $(A_{\infty}^{\fid},\Phi_{\infty}^{\fid})$ up to some error which decays exponentially in $\rho$, so we may instead work with the operators $D_{\infty}^2$  induced by the latter.  Then (switching now back to the variable $r$) the above $L^2$ orthogonal decomposition of the Hilbert space $H$ into the invariant subspaces $H_\ell^{\pm}$ can  be  refined further.  Namely, we have the $L^2$ orthogonal splitting $H_\ell^{\pm}=H_\ell^{\pm+}\oplus H_\ell^{\pm-}$, where $H_\ell^{\pm+}$ is the subspace of  those maps in $H_\ell^{\pm}$ which commute with $\Phi_{\infty}^{\fid}$, and $H_\ell^{\pm-}$ is its orthogonal complement. To be specific, $H_\ell^{++}$ is spanned by the maps of the form,
\begin{multline*}
\begin{pmatrix}
\mu_\ell\\ \sigma_{\ell+2}\\\sigma_{-\ell+2} \\ \tau_{\ell+1}\\ \tau_{-\ell+1}
\end{pmatrix} =\begin{pmatrix}
0\\ \sigma_{\ell+2}\\\sigma_{-\ell+2} \\\sigma_{\ell+2}\\ \sigma_{-\ell+2}
\end{pmatrix}  \\
= \left(0, \begin{pmatrix}0&  \sigma_{-\ell+2}e^{i(-\ell+2)\theta}+ \sigma_{\ell+2}e^{i(\ell+2)\theta} \\    \sigma_{-\ell+2}e^{i(-\ell+1)\theta}+ \sigma_{\ell+2}e^{i(\ell+1)\theta}&0 \end{pmatrix}dz\wedge d\bar z \right),
\end{multline*}
(which can be further decomposed into the two subspaces spanned by $\sigma_{\ell+2}$, respectively $\sigma_{-\ell+2}$),
and similar expressions can be derived for the three other subspaces. From \eqref{eq:opDtj+} we  read off that  $D_{\infty}^2$ acts on the subspace spanned by $(0, \sigma_{\ell+2},0,\sigma_{\ell+2},0)$ as
\begin{equation*}
D_{\infty,\ell}^+ \sigma_{\ell+2}=  -\frac{1}{2r^2}(r\partial_r)^2 \sigma_{\ell+2}+\frac{1}{2r^2} (\ell+\frac{3}{2})^2\sigma_{\ell+2}, 
\end{equation*}
respectively on that spanned by $(0, 0,\sigma_{-\ell+2},0,\sigma_{-\ell+2})$ as
\begin{equation*}
D_{\infty,\ell}^+ \sigma_{-\ell+2}=  -\frac{1}{2r^2}(r\partial_r)^2 \sigma_{-\ell+2}+\frac{1}{2r^2} (-\ell+\frac{3}{2})^2\sigma_{\ell+2}.
\end{equation*}
Therefore, the homogeneous equation  $D_{\infty,\ell}^+ \sigma_{\ell+2}=0$  admits the polynomially decaying solution $ \sigma_{\ell+2}=r^{-\ell-\frac{3}{2}}$ ($\ell\geq0$). Similarly, $D_{\infty,\ell}^+ \sigma_{-\ell+2}=0$  admits the  polynomially decaying solutions $\sigma_{2}=r^{-\frac{3}{2}}$,  $\sigma_{0}=\sigma_{1}=r^{-\frac{1}{2}}$, and  $\sigma_{-\ell+2}=r^{-\ell+\frac{3}{2}}$ ($\ell\geq3$). It is easily checked that the bounded solutions to the homogeneous equation contained in $H_\ell^{-+}$ likewise   decay to $0$ at rate $r^{-\frac{1}{2}}$ or faster. In contrast, bounded solutions  to the homogeneous equation contained in  $H_\ell^{\pm-}$ decay at an exponential rate to $0$. To see that, we  appeal to Lemma \ref{lem:subharmonicest}. With the pointwise inequality
\begin{equation*}
|[\ast\mu\wedge \Phi_{\infty}]|^2 + |[\ast\sigma\wedge\Phi_{\infty}^{\ast}]|^2 \geq C(|\mu|^2+|\sigma|^2)
\end{equation*}
being satisfied by all $(\mu,\sigma)\in  H_\ell^{\pm-}$ for some constant $C>0$, Eq.\ \eqref{eq:diffineqsubharm}  implies for the function $u=\frac{1}{2}(|\mu|^2+ \frac{3}{4}(|\sigma|^2)$   the differential inequality
\begin{equation*}
\Delta u\geq Cu.
\end{equation*}
From this inequality it is standard to conclude exponential decay of every bounded solution $(\mu,\sigma)$ to the homogeneous equation $D_{\infty,\ell}^{\pm}(\mu,\sigma) =0$,   the rate of decay   being independent of $\ell$.

\subsection{Estimates on the global solution}
   
We return to the discussion of the family of equations $D_t^j\xi_t=\eta_t$  on the surface $\Sigma$. Concerning the right-hand side $\eta_t$ we make the assumption that $\eta_t$ is supported in the disk $\D_{ct^{-\frac{2}{3}}}(p)$ about some fixed $p\in\mathfrak p$, and that it is  there of the form  $\eta_t(r,\theta)=\eta(\frac{8}{3}tr^{\frac{3}{2}})$ for some $t$-independent map $\eta$.

\begin{lemma}\label{lem:estmodelsolutions}
The unique solution $\xi_t$ of the equation $D_t^2\xi_t=\eta_t$ decomposes (not uniquely) as $\xi_t=u_t+r_t$ where  
\begin{itemize}
\item[(i)]
the map $u_t$ is supported in $\D(p)$ and $u_t(r,\theta)=t^{-\frac{4}{3}}u(ctr^{\frac{3}{2}})$ for some $t$-independent smooth map $u$;
\item[(ii)]
the map $r_t$ satisfies the estimates $\|r_t\|_{L^2(\Sigma)}\leq Ct^{-\frac{5}{3}}$ and 
\begin{equation*}
\sup_{x\in \D(p_1)}|r_t(x)|\leq Ct^{-\frac{5}{3}}
\end{equation*}
for some $t$-independent constant $C$ and every $p_1\in\mathfrak p$.
\end{itemize}
An analogue statement holds with $D_t^2$ replaced by the operator $D_t^0$.
\end{lemma}

\begin{proof}
The function $u_t$ is constructed as an approximate solution to the equation $D_t^2\xi_t=\eta_t$ as follows. On $\C$ we consider the  equation   $r^2D_t^2w_t=r^2\eta_t$. By the  discussion in \textsection  \ref{subsect:localanalysis}, using that  $r^2=ct^{-\frac{4}{3}} \rho^{\frac{4}{3}}$ decays at rate $t^{-\frac{4}{3}}$, it admits a solution $w_t$ of the form $w_t(r)=t^{-\frac{4}{3}} w(\rho)$ for some $t$-independent function $w$. Furthermore, this function $w$  can be chosen to decay at least at the polynomial rate $\rho^{-\frac{1}{3}}$ on the interval $[C,\infty)$ (with $C$ sufficiently large but fixed). 
We choose a smooth cutoff function $\chi_t\colon[0,\infty]\to[0,\infty)$ such that $\chi(\rho)=1$ for $0\leq \rho\leq Ct$ and the support of $\chi_t$ is contained in $[0,2Ct]$, and that furthermore $t|\chi_t'(\rho)|+t^2|\chi_t''(\rho)|\leq C$ for all $\rho$. Then we define the smooth map $u_t$ on $\Sigma$ by
\begin{equation*}
u_t(r,\theta)=\chi_t(\rho)t^{-\frac{4}{3}}w(\rho)
\end{equation*}
for $(r,\theta)\in\D(p)$, and continue it by zero outside $\D(p)$. By construction, the function $u_t$ is an approximate solution to the equation  $D_t^2\xi_t=\eta_t$ on $\Sigma$, with error  $S_t:=D_t^2u_t-\eta_t$  supported on the annulus $A(p)=\D_{2}(p)\setminus \D(p)$ of inner radius $1$ and outer radius $2$ around $p$. The function $S_t$ can be estimated as follows. Writing $D_t^2$ as $D_t^2=t^{\frac{4}{3}}\rho^{\frac{2}{3}}\big(-\frac{1}{\rho^2}(\rho\partial_{\rho})^2+M\big)$, where $M$ is some operator which does not involve derivatives with respect to $r$,  it follows that
\begin{eqnarray*}
S_t&=& t^{-\frac{4}{3}}\chi_t D_t^2w-\eta_t\\
&&+ \rho^{\frac{2}{3}}\Big(  -\frac{1}{\rho^2}(\rho\partial_{\rho})^2\chi_t\cdot w-\frac{1}{\rho^2}\rho\partial_{\rho}\chi_t\cdot \rho\partial_{\rho}w  \Big)\\
&=& \rho^{\frac{2}{3}}\Big(-\frac{1}{\rho^2}(\rho\partial_{\rho})^2\chi_t\cdot w- \partial_{\rho}\chi_t\cdot  \partial_{\rho}w  \Big).
\end{eqnarray*}
By the properties of the cutoff function $\chi_t$ and with $w$ decaying as $\rho^{-\frac{1}{3}}$, the term in the bracket admits the pointwise bound
\begin{align*}
\left|\frac{1}{\rho^2}(\rho\partial_{\rho})^2\chi_t\cdot w+\frac{1}{\rho}\partial_{\rho}\chi_t\cdot  \rho\partial_{\rho}w   \right|\leq C( t^{-2}\rho^{-\frac{1}{3}}  +\frac{1}{\rho} t^{-1} |\rho\partial_{\rho}w | )\\\leq Ct^{-2}\cdot t^{-\frac{1}{3}},
\end{align*}
where in the last step we used that $\operatorname{supp}\chi_t'\subseteq[Ct, 2Ct]$. On the other hand, the prefactor $ \rho^{\frac{2}{3}}$ grows like $t^{\frac{2}{3}}$ so that $S_t$ satisfies the pointwise bound $|S_t|\leq Ct^{-\frac{5}{3}}$. We are thus left with the equation $D_t^2r_t=S_t$, where $S_t$ is supported on the annulus $A(p)$. By the uniform boundedness of the family of maps $G_t^2\colon L^2(\Sigma)\to L^2(\Sigma)$ (cf.\ Corollary \ref{cor:uniformupperbound}) the asserted $L^2$ bound on $r_t$ follows. Concerning the claimed pointwise estimate we use that $r_t=(\mu_t,\sigma_t)$  satisfies the equation $D_t^2r_t=0$ on $\D(p_1)$. Thanks to Lemma \ref{lem:subharmonicest} this implies the  subharmonicity of the function $u_t=\frac{1}{2}(|\mu_t|^2+\frac{3}{4}|\sigma_t|^2)$ on that disk. An application of Lemma \ref{lem:unifC0bounddisks} then yields the result.  
\end{proof}

Along the same line of argument, we  obtain uniform bounds for the gauge correction term resulting from the solution $\xi_t$ of Eq.\ \eqref{eq:gaugefixeq} with right-hand side $E_t$.

\begin{lemma}\label{lem:refinedestgaugecorr}
Let $\xi_t$ denote the solution of the equation $D_t^0\xi_t=  E_t$, where $E_t$ is   as in Proposition \ref{prop:errorest}. It decomposes (not uniquely) as $\xi_t=u_t+r_t$ where  
\begin{itemize}
\item[(i)]
the map $u_t$ is supported in $\D(p)$ and $u_t(r,\theta)=t^{-\frac{1}{3}}u(ctr^{\frac{3}{2}})$ for some $t$-independent smooth map $u$;
\item[(ii)]
the map $r_t$ satisfies the estimate
\begin{equation}\label{eq:expdecayest1}
\|r_t\|_{C^0(\Sigma)}\leq Ce^{-\beta t}
\end{equation}
for $t$-independent  constants $\beta,C>0$.
\end{itemize}
Furthermore, there is a function $v=v(\rho,\theta)$ such that the gauge correction term $i_t\xi_t=i_tu_t+i_tr_t$ satisfies $|i_tu_t(r,\theta) |=t^{\frac{1}{3}} v(\frac{8}{3}tr^{\frac{3}{2}},\theta)$ for all $(r,\theta)\in\D(p)$  and 
\begin{equation}\label{eq:expdecayest2}
\|i_tr_t\|_{C^0(\Sigma)}\leq Ce^{-\beta t}.
\end{equation}
\end{lemma}

\begin{proof}
Since    $  E_t(r)=t   E(\rho)$ for some $t$-independent function $  E$, we are in the setup of Lemma \ref{lem:estmodelsolutions} and therefore obtain   statement (i) by an analogue line of argument. Statement (ii) also follows immediately; the sharper exponential decay we are claiming here is due the fact that $E$ is  a section of the subbundle of diagonal endomorphisms for which the discussion at the end of \textsection \ref{subsect:localanalysis} yields exponentially decaying solutions. Indeed, since $E_t$ is diagonal, so is $u_t$. Now the action of the operator $D_t^0$ on diagonal endomorphisms is given by
\begin{equation*}
D_t^0\begin{pmatrix} u&0\\0&-u\end{pmatrix}=\begin{pmatrix} \Delta u+ 8t^2r\cosh(2h_t)u&0\\0&-\Delta u- 8t^2r\cosh(2h_t)u\end{pmatrix},
\end{equation*} 
which since $8t^2r\cosh(2h_t)\geq C>0$ is uniformly positive admits exponentially decaying solutions to the equation $D_t^0\xi_t=0$ on $\C$. From here we may proceed as in the proof of Lemma \ref{lem:estmodelsolutions} and let $r_t$ be   the solution of the equation  $D_t^0r_t=S_t$, where the function $S_t$ is defined as before, but is now exponentially decaying in $t$. It was shown in \cite{msww14} that the norm of the operator $(D_t^0)^{-1}\colon L^2(\Sigma)\to H^2(\Sigma)$ is growing at an at most polynomial rate in $t$, which implies that $\|r_t\|_{H^2(\Sigma)}$ is exponentially decaying in $t$. Now a standard bootstrap argument shows exponential decay with respect to any $C^k$ norm, implying the estimates \eqref{eq:expdecayest1} and \eqref{eq:expdecayest2}. Concerning the term $i_tu_t=(d_{A_t^{\fid}}u_t,t[\Phi_t^{\fid}\wedge u_t])$ we substitute $\rho=\frac{8}{3}tr^{\frac{3}{2}}$ and $r\partial_r=\frac{3}{2}\rho\partial_{\rho}$ to obtain that 
\begin{equation*}
d_{A_t^{\fid}}u_t =  t^{-\frac{1}{3}}\left(  ct^{\frac{2}{3}} \rho^{\frac{1}{3}}\partial_{\rho}u\, d\rho+  \partial_{\theta}u\, d\theta\right).
\end{equation*}
Since $|d\theta|= r^{-1}= ct^{\frac{2}{3}}\rho^{-\frac{2}{3}}$, it follows that $d_{A_t^{\fid}}u_t $ grows like $t^{\frac{1}{3}}$. Now the second term is of the form
\begin{equation*}
t[\Phi_t^{\fid}\wedge u_t] = [tr^{\frac{1}{2}}\psi,  t^{-\frac{1}{3}} u]\, dz
\end{equation*}
for some endomorphism $\psi=\psi(\rho,\theta)$, from which the estimate on $|i_tu_t(r,\theta)|$ follows. 
\end{proof}

\section{Proof of the main theorem}

To finally show the asserted estimates on the sectional curvatures of $G_{L^2}$ we fix a simple holomorphic quadratic differential $q$ and let $t\mapsto S_t^{\appr}\in \mathcal M^{\appr}$ denote as in \textsection \ref{subsect:hitchinsection} the approximation to the Hitchin section along the ray $t\mapsto t^2q$. We further specify a family $\{X_t,Y_t\}$ of tangent two-frames of $\mathcal M^{\appr}$ along the path  $t\mapsto S_t^{\appr}$   in the following way.  We fix a pair $\{\dot q_1,\dot q_2\}$ of   holomorphic quadratic differentials. For each value of the parameter $t$, we let  $\{X_t,Y_t\}$
be the tangent vector in Coulomb gauge at $S_t^{\appr}$ induced by $\{\dot q_1,\dot q_2\}$ in the sense of Definition \ref{def:inducedtangentvector}. Choosing  $\dot q_1,\dot q_2$ in such a way that the pair $X_{\infty}$, $Y_{\infty}$ is $L^2$ orthonormal yields  for sufficiently large $t$  an approximately $L^2$ orthonormal two-frame  $X_{t}$, $Y_{t}$.

\begin{proposition}
For $t$ sufficiently large there holds the estimate
\begin{equation*}
\left|1- \| X_t\|^2\|   Y_t\|^2+\langle   X_t,   Y_t\rangle^2\right|\leq Ct^{-\frac{1}{3}}
\end{equation*}
for some constant $C$ which does not depend on $t$.
\end{proposition}

\begin{proof}
The stated inequality follows straightforwardly from estimate \eqref{eq:convhorizontal}.
\end{proof}

We introduce the notation  $X_t=(t^{-1}\alpha_{1,t},\varphi_{1,t})$ and $Y_t=(t^{-1}\alpha_{2,t},\varphi_{2,t})$. The purpose of the next proposition is to derive explicit expressions for these tangent vectors on each disk $\D(p)$, where $p\in\mathfrak p$. These are similar to \eqref{eq:correctedhortangentphi} and \eqref{eq:correctedhortangalpha}, however the effect of the final gauge correction step in \textsection \ref{subsect:apprtangentspaces} has to be taken into account. On $\D(p)$, write $\dot q_j=\dot f_j\, dz^2$ for a holomorphic function $\dot f_j$, $j=1,2$.

\begin{proposition}\label{prop:structuretangentvect}
The family $X_t=(t^{-1}\alpha_{1,t},\varphi_{1,t})$ (and similarly $Y_t=(t^{-1}\alpha_{2,t},\varphi_{2,t})$)   has the following properties.
\begin{compactitem}[(i)]
\item 
There is a constant $\beta>0$ such that  on each disk $\D(p)$, with $\alpha_{1,t}=\alpha_{1,t}^{0,1}\,d\bar z-(\alpha_{1,t}^{0,1})^{\ast}\,dz$,
\begin{equation*}
\alpha_{1,t}^{0,1}=(R_t^1 \dot f+ R_t^2 \bar{\dot f}) \begin{pmatrix}1&0\\0&-1\end{pmatrix}+\mathcal O(e^{-\beta t})
\end{equation*}
and
\begin{equation*}
\varphi_{1,t}=\begin{pmatrix}0&S_t\dot f\\
T_t\frac{\dot f}{f}&0\end{pmatrix}dz+\mathcal O(e^{-\beta t}).
\end{equation*}
Here 
$R_t^1$, $R_t^2$, $S_t$ and $T_t$ are functions of the radial variable $r$ of the form $R_t^1(r)=R^1(\frac{8}{3}tr^{\frac{3}{2}})$ etc.
\item[(ii)]
There is a constant $C$ such that
\begin{equation*}
|t^{-1}\alpha_{1,t}(x)|+| \varphi_{1,t}(x)|\leq Ct^{\frac{1}{3}},
\end{equation*}
 for all $t\geq1$ and $x\in \D_{ct^{-\frac{2}{3}}}(p)$, and $|t^{-1}\alpha_{1,t}|+| \varphi_{1,t} -\varphi_{1,\infty}|\to0$    at an exponential rate in $t$ outside the union of disks   $\D_{ct^{-\frac{2}{3}}}(p)$, $p\in\mathfrak p$.
\end{compactitem} 
\end{proposition}

\begin{proof}
Direct calculation shows that the asserted properties are satisfied by  the maps  $\varphi_t$ as in \eqref{eq:correctedhortangentphi} and $\alpha_t$ as in \eqref{eq:correctedhortangalpha}. They continue to hold true for  $\alpha_{1,t}$ and $\varphi_{1,t}$ as follows from the estimates on the gauge correction term stated in Lemma \ref{lem:refinedestgaugecorr}.
\end{proof}

After these preparations we turn to the proof of the main result.

\begin{proof}{\bf{(Proof of Theorem \ref{thm:mainthm}).}}
For ease of notation we abbreviate $(X,Y)=(X_t,Y_t)$ and $G^j=G_t^j$. We consider the numerator and denominator in
\begin{equation}\label{eq:sectcurvaturedef}
K(\Pi(X,Y))=\frac{\langle R(X,Y)Y,X\rangle}{ \| X\|^2\|   Y\|^2-\langle   X,   Y\rangle^2}
\end{equation}
separately.
As for the numerator, we have from \eqref{eq:sectionalcurvature} the expression 
\begin{multline*}
\langle R(X,Y)Y,X\rangle= 3\langle G^0P_X^{\ast}Y,P_X^{\ast}Y\rangle_{\mathfrak g}\\
+\langle G^2Q_XX,Q_YY\rangle_{V}-\langle G^2Q_XY,Q_XY\rangle_{V}.
\end{multline*}
It suffices to show that each of the summands on the right-hand side is of the asserted form. We prove this for $\langle G^2Q_XX,Q_YY\rangle_{V}$, the two other cases being similar. We split
\begin{multline*}
\langle G^2Q_XX,Q_YY\rangle_{V}= \int_{\Sigma^{\operatorname{ext}}} \langle G^2Q_XX,Q_YY\rangle\operatorname{vol}_{\Sigma}\\
+\sum_{p\in\mathfrak p}\int_{\D(p)}\langle G^2Q_XX,Q_YY\rangle \operatorname{vol}_{\Sigma}=: I+\sum_{p\in\mathfrak p}II_p.
\end{multline*}
Concerning the term $I$, it follows from Proposition \ref{prop:structuretangentvect} that   $Q_YY$ is exponentially decaying as $t\to\infty$ on $\Sigma^{\operatorname{ext}}$, while  $ G^2Q_XX$ is polynomially bounded in $t$  by Corollary \ref{cor:uniformupperbound}. Hence $|I|\to0$ at an exponential rate as $t\to\infty$. We now fix  $p\in\mathfrak p$ and consider the first component $G^2Q_XX$ of term $II_p$. Using a partition of unity subordinate to the open covering of $\Sigma$ provided by the sets $\D(p')$, $p'\in\mathfrak p$ and $\Sigma^{\operatorname{ext}}$, we can write
\begin{equation*}
G^2Q_XX=G^2(\chi_{\operatorname{ext}} Q_XX)+G^2(\chi_p Q_XX) + \sum_{p'\neq p} G^2(\chi_{p'} Q_XX). 
\end{equation*}
Again, $\chi_{\operatorname{ext}} Q_XX$ is exponentially decaying in $t$ and can therefore disregarded in what follows. Furthermore,  Proposition \ref{prop:structuretangentvect} shows that $|\chi_{p'} Q_XX|$ grows at rate $t^{\frac{2}{3}}$  on $\D(p')$ and has the scaling properties as asserted in Lemma \ref{lem:estmodelsolutions}. Hence $G^2(\chi_{p'} Q_XX)=u_t+r_t$ with $u_t$ being supported in $\D(p')$ and $|r_t|\leq Ct^{-\frac{5}{3}}\cdot t^{\frac{2}{3}}$. Hence the pairing 
\begin{equation*}
\int_{\D(p)}\langle G^2(\chi_{p'} Q_XX),Q_YY\rangle\operatorname{vol}_{\Sigma}
\end{equation*}
decays to $0$ at rate $t^{-\frac{4}{3}}\cdot t^{-1}\cdot t^{\frac{2}{3}}=t^{-\frac{4}{3}}\cdot t^{-\frac{1}{3}}$, where the prefactor $t^{-\frac{4}{3}}$ comes from the area form $rdr\, d\theta$ and the fact that $|Q_YY|$ is exponentially decaying in the region where $r>ct^{-\frac{2}{3}}$. It remains to consider 
\begin{equation*}
\int_{\D(p)}\langle G^2(\chi_{p} Q_XX),Q_YY\rangle\operatorname{vol}_{\Sigma},
\end{equation*}
which we claim contributes the leading order term to $II_p$, and decays at rate $t^{-\frac{4}{3}}$. In fact, this follows again by applying  Lemma \ref{lem:estmodelsolutions}, which yields that $G^2(\chi_{p} Q_XX)=u_t+r_t$ with $|u_t|$ decaying at rate $t^{-\frac{2}{3}}$. Thus the  main contribution to $II_p$ is $\int_{\D(p)}\langle u_t,Q_YY\rangle\operatorname{vol}_{\Sigma}$. The integrand here admits a uniform pointwise bound in $t$  and is supported (up to an exponentially small term) in $\D_{ct^{-\frac{2}{3}}}(p)$, so that the overall rate of decay of the integral is $t^{-\frac{4}{3}}$. We finally consider  the Taylor expansion of $G^2(\chi_{p} Q_XX)$ around $p$. We notice that the constant term satisfies a pointwise bound by $Ct^{\frac{2}{3}}$, while all other terms decay at least like $t^{-\frac{2}{3}}$. It follows that the leading order term in $II_p$ (which we have seen decays as $t^{-\frac{4}{3}}$) is determined by the value of $G^2(\chi_{p} Q_XX)$ at $p$. This value  in turn is determined by the values of the holomorphic quadratic differentials $\dot f_1$ and $\dot f_2$ at $p$. As for the denominator in \eqref{eq:sectcurvaturedef}, the estimate of Proposition \ref{prop:structuretangentvect} shows that it expands into $1$ plus an error term which decays at rate $t^{-\frac{1}{3}}$ in $t$. Hence the leading order term of the quotient in \eqref{eq:sectcurvaturedef} is  of the asserted form. This completes the proof of the theorem.
\end{proof}

\end{document}